\documentclass[11pt,amstex]{amsart}
\makeatletter

\renewcommand{\tocsection}[3]{%
  \indentlabel{\@ifnotempty{#2}{\bfseries\ignorespaces#1 #2\quad}}\bfseries#3}
\renewcommand{\tocsubsection}[3]{%
  \indentlabel{\@ifnotempty{#2}{\ignorespaces#1 #2\quad}}#3}

\newcommand\@dotsep{4.5}
\def\@tocline#1#2#3#4#5#6#7{\relax
  \ifnum #1>\c@tocdepth 
  \else
    \par \addpenalty\@secpenalty\addvspace{#2}%
    \begingroup \hyphenpenalty\@M
    \@ifempty{#4}{%
      \@tempdima\csname r@tocindent\number#1\endcsname\relax
    }{%
      \@tempdima#4\relax
    }%
    \parindent\z@ \leftskip#3\relax \advance\leftskip\@tempdima\relax
    \rightskip\@pnumwidth plus1em \parfillskip-\@pnumwidth
    #5\leavevmode\hskip-\@tempdima{#6}\nobreak
    \leaders\hbox{$\m@th\mkern \@dotsep mu\hbox{.}\mkern \@dotsep mu$}\hfill
    \nobreak
    \hbox to\@pnumwidth{\@tocpagenum{\ifnum#1=1\bfseries\fi#7}}\par
    \endgroup
  \fi}
\AtBeginDocument{%
\expandafter\renewcommand\csname r@tocindent0\endcsname{0pt}
}
\def\l@subsection{\@tocline{2}{0pt}{2.5pc}{5pc}{}}
\makeatother
\usepackage{hyperref}
 \usepackage[capitalize]{cleveref}
\usepackage{enumerate,color}
\usepackage{amssymb}
\usepackage{mathtools}
\usepackage[utf8]{inputenc}
\usepackage[T1]{fontenc}
\usepackage[colorinlistoftodos]{todonotes}
\usepackage[mathscr]{eucal}
\usepackage{amssymb}
\usepackage{stackengine}

\usepackage[a4paper, hmargin={2.5cm,2.5cm}, centering]{geometry}

\usepackage{xcolor}

\newcommand{\mymarginpar}[1]{%
\vadjust{\smash{\llap{\parbox[t]{\marginparwidth}{#1}\kern\marginparsep}}}}

\newtheorem{lemma}{Lemma}[section]
\newtheorem{theorem}[lemma]{Theorem}
\newtheorem*{theorem*}{Theorem}

\newtheorem{proposition}[lemma]{Proposition}
\newtheorem*{proposition*}{Proposition}
\newtheorem{conjecture}{Conjecture}
\newtheorem{remark}[lemma]{Remark}

\theoremstyle{remark}

\theoremstyle{remark}

\theoremstyle{definition}
\newtheorem{problem}{Problem}

\newcommand{\E}{{\mathbb E}}

\newcommand{\N}{{\mathbb N}}

\newcommand{\R}{{\mathbb R}}
\newcommand{\T}{{\mathbb T}}
\newcommand{\Z}{{\mathbb Z}}

\newcommand{\norm}[1]{\left\Vert #1\right\Vert}
\newcommand{\nnorm}[1]{\lvert\!|\!| #1|\!|\!\rvert}

\theoremstyle{definition}

\newtheorem*{definition*}{Definition}
\newtheorem*{conjecture*}{Conjecture}
\newtheorem*{remark*}{Remark}
\newtheorem*{remarks*}{Remarks}
\newtheorem*{claim*}{Claim}

\newtheorem{example}{Example}

\def \e {\epsilon}
\def \Es {\overline{\mathbb{E}}}

\def \h {\bold{h}}

\def \p {\bold{p}}
\def \q {\bold{q}}
\def \u {\bold{u}}

\def \w {\bold{w}}

\def \I {\mathcal{I}}

\def \vl {\varlimsup_{N\to \infty}}

\begin{document}

\title{Joint ergodicity for functions of polynomial growth}

\author{Sebasti{\'a}n Donoso, Andreas Koutsogiannis and Wenbo~Sun}

\address[Sebasti{\'a}n Donoso]{\textsc{Departamento de Ingenier\'{\i}a Matem{\'a}tica and Centro de Modelamiento Matem\'atico, Universidad de Chile \& IRL 2807 - CNRS, Beauchef 851, Santiago, Chile.}} \email{sdonoso@dim.uchile.cl}

\address[Andreas Koutsogiannis]{
Department of Mathematics, Aristotle University of Thessaloniki, Thessaloniki, 54124, Greece}
\email{akoutsogiannis@math.auth.gr}

\address[Wenbo Sun]{\textsc{Department of Mathematics, Virginia Tech, 225 Stanger Street, Blacksburg, VA, 24061, USA}}
\email{swenbo@vt.edu}

\thanks{The first author was supported by ANID/Fondecyt/1200897 and Centro de Modelamiento Matemático (CMM), FB210005, BASAL funds for centers of excellence from ANID-Chile.}



\maketitle

\begin{abstract}  
We provide necessary and sufficient conditions for joint ergodicity results for systems of commuting measure preserving transformations for an iterated Hardy field function of polynomial growth. 
Our method builds on and improves recent techniques due to Frantzikinakis and Tsinas, who dealt with multiple ergodic averages along Hardy field functions; it also enhances an approach introduced by the authors and Ferr\'e Moragues to study polynomial iterates. The more general expression, in which the iterate is a linear combination of a Hardy field function of polynomial growth and a tempered function, is studied as well.
\end{abstract}

\makeatletter
\providecommand\@dotsep{10}
\makeatother

\section{Introduction}

A central problem in ergodic theory is the study of multiple ergodic averages of the form
\begin{equation}\label{E:main_mea}
\frac{1}{N}\sum_{n=1}^N T_1^{a_1(n)}f_1\cdot\ldots\cdot T_d^{a_d(n)}f_d,
\end{equation}
where $(X,\mathcal{B},\mu,T_1,\ldots,T_d)$ is a system (that is, $(X,\mathcal{B},\mu)$ is a Borel probability space and for all $1\leq i\leq d$, $T_i\colon X\to X$ is a measurable, measure preserving transformation, i.e., $\mu(T_i^{-1}A)=\mu(A)$ for all $A\in \mathcal{B}$),  
 for each $1\leq i\leq d,$ $(a_i(n))_n$ is an appropriate integer-valued sequence, and $f_i$ is a bounded function; for a positive integer $n,$ 
$T^n$ denotes  the composition $T\circ \dots\circ T$ of $n$ copies of $T$, and $Tf(x)\coloneqq f(Tx),$ $x\in X.$ In particular, we are interested in the ($L^2(\mu)$) norm limiting behaviour, as $N\to\infty,$ of \eqref{E:main_mea} for various $a_i$'s, and commuting $T_i$'s (i.e., $T_iT_j=T_jT_i$). Our study deals with commuting and invertible $T_i$'s. 

Furstenberg's celebrated result (\cite{Furst}), i.e., proving Szemer\'edi's theorem (that each dense subset of natural numbers contains arbitrarily long arithmetic progressions) by studying 
\eqref{E:main_mea} for $T_i=T$ and $a_i(n)=in,$ revolutionized the area, leading to far-reaching extensions of Szemer\'edi's theorem and various other profound results.  For many of the latter results, the only known proofs are the ergodic theoretic ones.

 For $d=1$ and $a_1(n)=n$ in \eqref{E:main_mea}, von Neumann's mean ergodic theorem characterizes ergodicity:\footnote{T is \emph{ergodic} if $A\in\mathcal{B},T^{-1}A=A,$ implies that $\mu(A)\in \{0,1\}.$} $T$ is ergodic if, and only if, $\frac{1}{N}\sum_{n=1}^N T^n f\to \int f\;d\mu$  as $N\to\infty.$ For $T=T_i$ \emph{weakly mixing} (\emph{w.m.} for short) (i.e., $T\times T$ is ergodic), and $a_i(n)=in$, Furstenberg showed (again in \cite{Furst}) that \eqref{E:main_mea} converges to $\prod_{i=1}^d \int f_i\;d\mu$. This result was extended in \cite{wmpet} by Bergelson for the case $T_{1}=\dots=T_{d}$ being w.m. and $a_i$ being essentially distinct integer polynomial iterates.\footnote{ $p\in \mathbb{Q}[x]$ is an \emph{integer polynomial} if $p(\mathbb{Z})\subseteq \mathbb{Z}$; $\{p_1,\ldots,p_d\}$ are \emph{essentially distinct} if $p_i, p_i-p_j$ are non-constant for all $i\neq j.$}  Because of the aforementioned results, we call (for ergodic systems) $\prod_{i=1}^d \int f_i\;d\mu$ the ``expected limit''. So, naturally, one defines the following notion.
 
  \begin{definition*}
Let $(X,\mathcal{B},\mu)$ be a Borel probability space, and $(S_{1}(n))_{n},\dots,(S_{d}(n))_{n}$ be sequences of measure preserving transformations on $X$. We say that $(S_{1}(n))_{n},\dots,(S_{d}(n))_{n}$ are \emph{jointly ergodic (for $\mu$)}, if for all functions $f_1,\ldots,f_d\in L^\infty(\mu)$ we have 
    \begin{equation}\label{E:correct_limit}
        \lim_{N\to\infty}\frac{1}{N}\sum_{n=1}^N S_{1}(n)f_1\cdot\ldots\cdot S_{d}(n)f_d=\int f_1\;d\mu\cdot\ldots\cdot\int f_k\;d\mu,
    \end{equation}
    where the convergence takes place in $L^2(\mu).$ When $d=1,$ we simply say that the sequence $(S_{1}(n))_{n}$ is \emph{ergodic}.\footnote{ Here, by saying that we have joint ergodicity, we mean that the limit in \eqref{E:correct_limit} exists, and it is the expected one.}
 \end{definition*}

The first characterization of joint ergodicity 
is due to Berend and Bergelson \cite{BB}:

\begin{theorem*}[\cite{BB}]
 Let $(X,\mathcal{B},\mu,T_{1},\dots,T_{d})$ be a system with commuting and invertible transformations. 
 Then $(T_{1}^n)_n,\dots,(T_{d}^n)_n$ are jointly ergodic for $\mu$ if, and only if, both of the following conditions are satisfied:
	\begin{itemize}
		\item[(i)] $T_{i}T^{-1}_{j}$ is ergodic for $\mu$ for all $1\leq i,j\leq d,$ $i\neq j$; and
		\item[(ii)] $T_{1}\times \dots\times T_{d}$ is ergodic for $\mu^{\otimes d}$.
	\end{itemize}
\end{theorem*}

This theorem, for $T_i=T^i,$ where $T$ is a w.m. transformation, implies Furstenberg's w.m. convergence result. 
A few years ago, Bergelson, Leibman, and Son showed (in \cite{BLS}) the following result for generalized linear functions.\footnote{ A \emph{generalized linear function} $\varphi\colon \mathbb{N}\to\mathbb{Z}$ is a function of the form $\varphi(n)=[an]+e_n,$ where 
 $[\cdot]$ is the integer value, or floor, function and $e_n$ is some special, bounded, integer-valued error term.}

\begin{theorem*}[\cite{BLS}]
 Let $(X,\mathcal{B},\mu,T_{1},\dots,T_{d})$ be a system with commuting and invertible transformations, 
 and $\varphi_1,\ldots,\varphi_d$ be generalized linear functions. Then $(T_{1}^{\varphi_1(n)})_n,\dots,(T_{d}^{\varphi_d(n)})_n$ are jointly ergodic for $\mu$ if, and only if, both of the following conditions are satisfied:
	\begin{itemize}
		\item[(i)] $\Big(T_{i}^{\varphi_i(n)}T_{j}^{-\varphi_j(n)}\Big)_n$ is ergodic for $\mu$ for all $1\leq i,j\leq d,$ $i\neq j$; and
		\item[(ii)] $\Big(T_{1}^{\varphi_1(n)}\times \dots\times T_{d}^{\varphi_d(n)}\Big)_n$ is ergodic for $\mu^{\otimes d}$.
	\end{itemize}
\end{theorem*}

 This result, for $T_i=T$ and $\varphi_i(n)=[\alpha_i n],$ where $T$ is w.m. and the $\alpha_i$'s are distinct real numbers, extends Furstenberg's w.m. convergence result.

Seeing the similarities of the last two results, it is reasonable to state the following problem.

\begin{problem}[Joint ergodicity problem]\label{Prb:1}
Let $(X,\mathcal{B},\mu,T_{1},\dots,T_{d})$ be a system with commuting and invertible transformations. Find classes of integer-valued sequences $a_1,\ldots,a_d$ so that $(T_{1}^{a_1(n)})_n,\dots,(T_{d}^{a_d(n)})_n$ are jointly ergodic for $\mu$ if, and only if, both of the following conditions are satisfied:
	\begin{itemize}
		\item[(i)] $\Big(T_{i}^{a_i(n)}T_{j}^{-a_j(n)}\Big)_n$ is ergodic for $\mu$ for all $1\leq i,j\leq d,$ $i\neq j$; and
		\item[(ii)] $\Big(T_{1}^{a_1(n)}\times \dots\times T_{d}^{a_d(n)}\Big)_n$ is ergodic for $\mu^{\otimes d}$.
	\end{itemize}
\end{problem}

Answering a question due to Bergelson, we showed in \cite{DKS} that the answer to Problem~\ref{Prb:1} is affirmative when $a_{1},\dots,a_{d}$ are equal to the same integer polynomial. This result was later generalized in \cite{dfks} to the case where all the $a_{i}$'s  are polynomials that can be grouped in a way such that polynomials in different groups have different degrees and each two polynomials in the same grouping are multiples of each other.  In two recent papers \cite{FraKu,FraKu2}, Frantzikinakis and Kuca showed that  the answer to Problem~\ref{Prb:1} is affirmative for all integer polynomials (modulo mild necessary conditions) $a_{1},\dots,a_{d}$.

In this paper, we extend the study of \cref{Prb:1} to functions $a_{1},\dots,a_{d}$ beyond polynomials. In literature, the multiple ergodic averages  (\ref{E:main_mea}) with $a_{i}$'s being Hardy field functions (see Section~\ref{Sec_2} for definition) of polynomial growth\footnote{ A function $h$ has \emph{polynomial growth} if it satisfies $h(x)\ll x^d$ for some $d\in \mathbb{N},$ where, for two functions $a,b:(x_0,\infty)\to\mathbb{R},$ we write  $a\ll b$ if there exists a universal constant $C>0$ so that $|a(x)|\leq C |b(x)|$ for all $x.$}  
has been studied extensively (see for instance \cite{BMR, frahardycommutingsz,Fra3,FrW, Flo, Ts}). However, to the best of our knowledge, joint ergodicity results for such functions for systems with commuting transformations have not been obtained in the past. 
By \cite[Lemma~A.3]{Flo}, every Hardy field function $h$ of polynomial growth can be written as 
\[h(x)=s_h(x)+p_h(x)+e_h(x),\]
where $s_h$ is a strongly non-polynomial Hardy field function,\footnote{  By this we mean that $s_h$ is a Hardy field function and that, for some non-negative integer $i,$ it satisfies $x^{i}\prec s_h(x)\prec x^{i+1},$ where for two functions $a,b:(x_0,\infty)\to\mathbb{R},$ we write $a\prec b$ if $|a(x)|/|b(x)|\to 0$ as $x\to\infty.$} $p_h$ is a polynomial and $e_h(x)\to 0$ as $x\to \infty.$ 
Under the additional, natural, assumption  $\log \prec s_h,$\footnote{ This is a natural assumption for convergence results, in the sense that for strongly non-polynomial Hardy field functions it implies equidistribution (see, e.g., \cite{boshernitzaneqd, frahardycommutingsz, Fra3}).} we have:  

\begin{theorem}\label{t1_2}
Let $(X,\mathcal{B},\mu,T_{1},\dots,T_{d})$ be a system with commuting and invertible transformations, and $h$ be a Hardy field function from $\mathcal{H}$\footnote{$\mathcal{H}$ is a Hardy field that will also be defined in Section~\ref{Sec_2}.} of polynomial growth, with $\log \prec s_h$. Then $(T_{1}^{[h(n)]})_n,\dots,(T_{d}^{[h(n)]})_{n}$ are jointly ergodic for $\mu$ if, and only if, both of the following conditions are satisfied:
	\begin{itemize}
		\item[(i)] $((T_{i}T^{-1}_{j})^{[h(n)]})_{n}$ is ergodic for $\mu$ for all $1\leq i,j\leq d,$ $i\neq j$; and
		\item[(ii)] $((T_{1}\times \dots\times T_{d})^{[h(n)]})_{n}$ is ergodic for $\mu^{\otimes d}$.
	\end{itemize}
\end{theorem}

The following are some examples of functions $h$ that we can deal with:
\[x\log x,\footnote{This example was proven to be good for convergence very recently in \cite{Ts}; what is crucial here is the fact that one can approach this function by variable polynomials. Its derivative is, modulo the constant $1,$ equal to $\log x,$ a ``bad'' for convergence function as it fails to be equidistributed.}\;\; x^e \log^2 x + x^{17}. \]
We believe that the assumption $\log \prec s_h$ can be lifted to $\log \prec h,$ and that, even though we are dealing with a single sequence in Theorem~\ref{t1_2}, something more general holds:

\begin{conjecture}\label{conj}
Problem~\ref{Prb:1} holds for functions $h_1, \ldots, h_d$ from the Hardy field $\mathcal{H}$ of polynomial growth with $\log \prec h_i,$ $1\leq i\leq d.$ 
 \end{conjecture}

Our method in fact yields a result more general than Theorem~\ref{t1_2} (see Theorem~\ref{t1_3}), in which the iterated function is the sum of a Hardy field and a tempered function (see Section~\ref{Section:7} for the definition) of different growth.

\subsection{Strategy of the paper}

In literature, the study of multiple ergodic averages  (\ref{E:main_mea}) mainly focuses on polynomial, Hardy field and tempered functions (and combinations of them).\footnote{  A (far from complete) list here is the following: \cite{wmpet, bk, BL0, BMR, AA, CFH, dfks, DKS, DKS2, frahardycommutingsz, FA, Fra3, FraKu, FraKu2, FrW, KK, Kifer, koutsogiannis2, koutsogiannis, K1, Ts, W12}.} In fact, these are  the only known classes of functions with the following property: If a function $f$ of ``degree $k$'' (meaning that $n^{k}\ll f(n)\prec n^{k+1}$)  belongs to the class, then its derivative $f'$ not only belongs to the same class but also is of ``degree $k-1$''
(meaning that $n^{k-1}\ll f'(n)\prec n^{k}$). For this very reason, the usual approach to studying the corresponding (\ref{E:main_mea}), for such classes of functions, is to reduce its complexity by using variants of the van der Corput lemma combined with variations of the PET induction (see Section~\ref{Sec_PET}).
 
During the past few years there is an interest in the class of variable polynomial sequences. The study of multiple ergodic averages with such iterates is an interesting new topic on its own, with open problems (see, e.g., \cite{Fra4,koutsogiannis2}) and results which have led to variable variations of classical theorems  (see \cite{Fra3,GT,Kifer,koutsogiannis2,Ts}). Most importantly, it provides an additional tool that can be used for the study of averages with iterates coming from the suitable classes of functions mentioned above.  
More specifically, for iterates which are Hardy field functions (or, even more generally, ``smooth enough'' functions--see Section~\ref{Section:7}), we can alternatively approach them by variable polynomials of bounded degrees first, and then run the PET induction via the van der Corput's lemma, on the polynomials. This alternative approach can also treat iterates which cannot be treated by the first one, such as $x\log x.$ 

Our strategy is to study the joint ergodicity problem by incorporating the second approach mentioned above with the machinery created in our previous works \cite{dfks,DKS}. 
To achieve this, we first approximate the iterates in the multiple ergodic average of interest by variable polynomials. Then, we extend the concepts of PET tuples and vdC operations in \cite{dfks,DKS} for variable polynomials, and use them to bound the 
stated average by an average of ergodic averages with linear iterates (on ``short intervals''). Finally, we deduce the desired result by using
 \cite[Proposition~5.2]{dfks}, which is the central part of \cite{dfks,DKS} where concatenation theorems  from \cite{TZ} were crucially used.
To achieve this, we need to make several adaptions to the approaches in these works.
To be more precise, we need to extend the concepts of PET tuples and vdC operations to variable polynomials, i.e., of the form $[p_N(n)],$ and generalize certain seminorm estimates for multiple ergodic averages with such iterates. 

As we already mentioned, \cite[Lemma~A.3]{Flo} implies that every Hardy field function $h$ of polynomial growth can be written as 
\[h(x)=s_h(x)+p_h(x)+e_h(x),\]
where $s_h$ is a strongly non-polynomial Hardy field function, $p_h$ is a polynomial and $e_h(x)\to 0$ as $x\to \infty.$
We will work with such functions, under the assumption that $\log \prec s_h.$
In Section~\ref{Sec_2}, Proposition~\ref{P:HostSemi}, we show that if a function $a$ can be written as
\[a(N+r)=p_N(r)+e_{N,r},\] with $e_{N,r}\ll 1,$ for every positive integer $N$ and $0\leq r\leq L(N),$ where $L$ is an appropriate positive Hardy field function satisfying $1\prec L(x)\prec x,$ and $(p_N)_N$ is a variable polynomial sequence, then, to study the initial multiple averages with iterates $[a(n)]$ along $1\leq n\leq N,$ it suffices to study the corresponding averages with iterates $[p_N(n)]$ along $0\leq n\leq L(N).$ We also prove in Proposition~\ref{P:cov} a change of variables statement, which shows that if the sequence of variable polynomials $p_N$ has ``special'' leading coefficients, then we can transform it to one with leading coefficients $1.$\footnote{ Another indication that variable polynomial sequences with leading coefficients $1$ form a good class of (variable) polynomials to deal with, is also revealed in \cite{Fra3}, where Frantzikinakis showed that, for a single transformation, multiple ergodic averages with such iterates have the nilfactors as characteristic factors.}
In Section~\ref{Sec_PET}, we extend the results on PET induction from \cite{dfks,DKS} to variable polynomials, which is used to reduce the complexity of multiple ergodic averages with such iterates, and eventually, via Lemma~\ref{reduction}, reduce the problem to the base case, namely, the linear one. Then, in Section~\ref{Sec_HK}, we provide, in Proposition~\ref{polytolinear2}, a Gowers-Host-Kra-type seminorm upper bound for multiple ergodic averages for certain linear variable polynomials, and in Theorem~\ref{polytolinear} a bound for variable polynomials of leading coefficient $1$ (using the inductive scheme of Section~\ref{Sec_PET}). In Section~\ref{Sec_MR}, we show that, for the functions $h$ we deal with, we can combine all the ingredients proved in the previous sections to deduce our main result, Theorem~\ref{t1_2}, via Theorem~\ref{T:upper bound}, \cite[Corollary~2.5]{DKS}, and \cite[Theorem~1.1]{AA}. 
Finally, in Section~\ref{Section:7}, we explain how our method can be used to study more general 
 iterates, beyond Hardy field functions (see Theorem~\ref{t1_3} which extends Theorem~\ref{t1_2}).\footnote{ Our approach can deal with more general iterates, namely functions $a$ of the form
\[a(x)=c_1h(x)+c_2t(x),\]
where $(c_1,c_2)\in \mathbb{R}^2\setminus\{(0,0)\},$ $h\in\mathcal{H}$ is as before, and $t\in \mathcal{T}$ is a tempered function, with different growth rates, which also satisfy some natural growth rate-related assumptions (see Section~\ref{Section:7}).}

\subsection{Notation}
We denote with $\mathbb{N},$ $\mathbb{N}_0,$ $\mathbb{Z},$ $\mathbb{Q},$ $\mathbb{R}$, $\mathbb{C}$ and $\mathbb{S}^{1}$ the set of positive integers, non-negative integers, integers, rational numbers, real numbers, complex numbers and complex numbers of modulus 1 respectively. If $X$ is a set, and $d\in \mathbb{N}$, $X^d$ denotes the Cartesian product $X\times\cdots\times X$ of $d$ copies of $X$. For $M,N\in\mathbb{Z}$ with $M\leq N$, let $[M,N]\coloneqq\{M,M+1,\dots,N\}$; we also define $[N]\coloneqq\{0,\dots,N-1\}$. We denote by $e_i$ the $i$-th standard unit vector, which has $1$ as its $i$-th coordinate and $0$ elsewhere. 

Let  $(a(n))_{n}$ be a sequence of complex numbers, or a sequence of measurable functions on a probability space $(X,\mathcal{B},\mu),$ indexed by the set of natural numbers. Throughout this article, we use the following notation for averages:

$\mathbb{E}_{n\in A}a(n)  \coloneqq  \frac{1}{|A|} \sum_{n \in A } a({n}),  \; \text{ where A is a finite subset of} \;\mathbb{Z};$

$\E_{n\in \Z}a(n)  \coloneqq \lim_{N\to \infty} \mathbb{E}_{n \in [-N,N] } a(n) \text{ if the limit exists};$

$\overline{\E}_{n\in \Z}a(n)  \coloneqq \vl \mathbb{E}_{n \in [-N,N] } a(n).$
\medskip

We also consider \emph{iterated} averages: Let $(a(h_{1},\dots,h_{s}))_{h_{1},\dots,h_{s}\in \Z}$ be a multi-parameter sequence. We let
\[\E_{h_{1},\dots,h_{s}\in \mathbb{Z}}a(h_{1},\dots,h_{s})\coloneqq \E_{h_{1}\in\mathbb{Z}}\ldots\E_{h_{s}\in\mathbb{Z}}a(h_{1},\dots,h_{s})\]
if the limit exists, we
adopt similar conventions for $\overline{\mathbb{E}}_{h_{1},\dots,h_{s}\in \mathbb{Z}}$ and for averages indexed in $\mathbb{N}$.

We end this section by recalling the notion of a system indexed by $(\mathbb{Z}^d,+),$ $d\in\mathbb{N}$.
We say that a tuple $(X,\mathcal{B},\mu,(T_{n})_{n\in \mathbb{Z}^d})$ is a \emph{$\mathbb{Z}^d$-measure preserving system} (or a \emph{$\mathbb{Z}^d$-system}) if $(X,\mathcal{B},\mu)$ is a probability space and, $T_{n}\colon X\to X,$ $n\in \mathbb{Z}^d,$ are measurable, measure preserving transformations on $X$ such that $T_{(0,\ldots,0)}={\rm id}$ and $T_{n}\circ T_{m}=T_{n+m}$ for all $n,m\in \mathbb{Z}^d$. Given $d$ commuting and invertible 
transformations $T_1,\ldots,T_d,$ we can naturally define a $\mathbb{Z}^d$-action as follows: \[T_{n}=T_1^{n_1}\cdot\ldots\cdot T_d^{n_d},\;\;n=(n_1,\ldots,n_d)\in \mathbb{Z}^d.\] So, we identify the $\mathbb{Z}$-action $T_i$ with $T_{e_i}$ for $1\leq i\leq d$.\footnote{Notice that we change our notation form $T^n,$ $n\in \mathbb{Z}$ to $T_{n},$ $n\in \mathbb{Z}^d$ when dealing with a $\mathbb{Z}$ or, respectively, a $\mathbb{Z}^d$-action to distinguish them.}   By slightly abusing the notation, we also refer to $(X,\mathcal{B},\mu,T_1,\ldots,T_d)$ as a $\mathbb{Z}^d$-system.
Let $H$ be a subgroup of $\mathbb{Z}^{d}$. We say that $H$ is \emph{ergodic} for a $\mathbb{Z}^d$-system $(X,\mathcal{B},\mu,(T_{n})_{n\in \mathbb{Z}^d})$ 
if for every $A\in\mathcal{B}$ such that $T_{g}A=A$ for all $g\in H$, we have that $\mu(A)\in \{0,1\}$. 
In particular, we say that $(X,\mathcal{B},\mu,(T_{n})_{n\in \mathbb{Z}^d})$ is \emph{ergodic} if $\mathbb{Z}^{d}$ is ergodic for the system. 
Finally, when it is clear, 
we will write $\norm{\cdot}_2$ instead of $\norm{\cdot}_{L^2(\mu)}$ and $\norm{\cdot}_{\infty}$ instead of $\norm{\cdot}_{L^{\infty}(\mu)}$.

 For a set of parameters $A$, and a potivite real number $\alpha,$ we write $O_{A}(\alpha)$ to denote a quantity that is $\leq C_A\cdot\alpha$ for some constant $C_A>0$ depending only on the parameters in $A$; if the constant $C$ is universal, we write $O(\alpha)$ instead.

\section{Reduction to variable polynomial iterates}\label{Sec_2}

We start with the defintion of Hardy field functions.
Let $B$ be the collection of equivalence classes of real valued functions defined on some halfline $(x_0,\infty),$ $x_0\geq 0,$ where two functions that eventually agree are identified. These equivalence classes are called \emph{germs} of functions.  A \emph{Hardy field} is a subfield of the ring $(B, +, \cdot)$ that is closed under differentiation.\footnote{We use the word \emph{function} when we refer to elements of $B$ (understanding that all the operations defined and statements made for elements of $B$ are considered only for sufficiently large values of $x\in \mathbb{R}$).} 

Usually, one deals with the class of logarithmico-exponential Hardy field functions, $\mathcal{LE}$, which can be handled more easily: $h$ is a \emph{logarithmico-exponential Hardy field function} if it is defined on some $(c,+\infty),$ $c\geq 0,$ by a finite combination of symbols $+, -, \times, \div, \sqrt[n]{\cdot}, \exp, \log$ acting on the real variable $x$ and on real constants (for more on Hardy field functions and in particular for logarithmico-exponential ones one can check \cite{frahardycommutingsz}, \cite{Fra3}, \cite{Har}). 
 
 As in \cite{Ts}, we will work on the Hardy field $\mathcal{H}$ which is closed under composition and compositional
inversion of functions, when defined (i.e., if $h_1, h_2\in \mathcal{H}$ with $\lim_{x\to\infty}h_2(x)=\infty,$ then $h_1\circ h_2,$ $h_2^{-1}\in \mathcal{H}$).\footnote{ Notice here that $\mathcal{LE}$ does not have this property but it is contained in the Hardy field of Pfaffian functions which does (\cite{Kho}).} 

 A function $h$ from $\mathcal{H}$ of polynomial growth has \emph{degree}  a non-negative integer $d_h\geq 0,$ if $x^{d_h}\ll h(x)\prec x^{d_h +1}$ (recall from the introduction that if $x^{d_h}\prec h(x)\prec x^{d_h+1},$ then $h$ is strongly non-polynomial).

As we mention before, our work concerns iterates involving families of variable polynomials.  A sequence of real \emph{variable polynomials} is a sequence of the form $(p_N(n))_{N,n}\subseteq \mathbb{R},$ where we assume that while the polynomials $p_N$ might depend on $N,$ their degrees do not.\footnote{ For a study on ``good'' variable polynomials, see \cite{koutsogiannis2}.}  The following are two examples of sequences of variable polynomials:
\[p_{N,1}(n)=\frac{n^{17}}{\sqrt{N}},\;\;\;p_{N,1}(n)=\Big(\frac{\sqrt{2}}{N^{e/\pi}}+\frac{N}{3}\Big)n^7-\frac{33}{\log N}n+1,\;\;\;N,n\in \N.\]
As in \cite{Ts}, the main idea in our setting is that we will approximate a given function, $a\in\mathcal{H},$ by ``good'' variable polynomials, $(p_N)_{N},$ in suitable intervals (with lengths that tend to infinity). Then, as reflected in the following proposition, to study multiple ergodic averages with iterates $[a(n)],$ it suffices to study some related weighted (with some bounded error terms as weights) averages  with iterates $[p_N(n)].$

\begin{proposition}\label{P:HostSemi}
Let $(X,\mathcal{B},\mu,T_{1},\dots,T_{d})$ be a system with commuting and invertible transformations.
			Let $a$ be a function and $L\in \mathcal{H}$ a positive function  with $1\prec L(x)\prec x.$
   Let $(p_{N})_{N}$ be a sequence of functions 
   such that for all $N\in \mathbb{N}$ and $0\leq r\leq L(N),$  
			\[a(N+r)=p_N(r)+e_{N,r},\;\; \text{with}\;\; e_{N,r}\ll 1.\] Assuming that 
			\[\limsup_{N \to \infty}\sup_{\vert c_{n}\vert\leq 1}\sup_{\Vert f_{2}\Vert_{\infty},\dots,\Vert f_{d}\Vert_{\infty}\leq 1}\norm{\mathbb{E}_{0\leq n\leq L(N)}c_{n}\prod_{i=1}^{d}T_{i}^{[p_N(n)]}f_{i}}^\kappa_{2}=0,\] for some $\kappa\in\mathbb{N}$ and $f_{1}\in L^{\infty}(\mu)$,  we have
			\[	\limsup_{N\to\infty}\norm{\mathbb{E}_{1\leq n\leq N}\prod_{i=1}^{d}T_{i}^{[a(n)]}f_{i}}_2=0\]
   for all $f_{2},\dots,f_{d}\in L^{\infty}(\mu)$.
		\end{proposition}
		
		\begin{proof}
		 To show the result, by \cite[Lemma~3.3]{Ts},\footnote{ We remark at this point that the assumption $L\in \mathcal{H},$ where $\mathcal{H}$ is a, closed under composition and compositional inversion of functions, Hardy field, is postulated exactly so we can use this statement. }  it suffices to show that
	 \begin{equation*}
	 \limsup_{R\to\infty}\mathbb{E}_{1\leq N\leq R}\norm{ \mathbb{E}_{N\leq n\leq N+L(N)}\prod_{i=1}^{d}T_{i}^{[a(n)]}f_{i}}^{\kappa}_{2}=0,
	 \end{equation*}
 hence, it suffices to show  
	 \begin{equation*}
	 \limsup_{N \to \infty}\norm{ \mathbb{E}_{N\leq n\leq N+L(N)}\prod_{i=1}^{d}T_{i}^{[a(n)]}f_{i}}^{\kappa}_{2}=0.
	 \end{equation*}
  Write $n=N+r$ for some $0\leq r\leq L(N)$.
	 Since $a(N+r)=p_N(r)+e_{N,r}$, then $[a(N+r)]=[p_N(r)]+\tilde{e}_{N,r,}$ $\tilde{e}_{N,r}\ll 1,$ hence the left hand side of the previous relation is equal to 
			\begin{equation}\label{ntaylor4n_new}
		\limsup_{N \to \infty}\norm{\mathbb{E}_{0\leq r\leq L(N)}\prod_{i=1}^{d}T_{i}^{[p_{N}(r)]+\tilde{e}_{N,r}}f_{i}}^{\kappa}_{2}.
		\end{equation}
		Since $\tilde{e}_{N,r}\ll 1$, 	
		by \cite[Lemma~3.2]{Ts}, (\ref{ntaylor4n_new}) is bounded by a constant multiple of the quantity
		\begin{equation*}
			\limsup_{N \to \infty}\sup_{\vert c_{r}\vert\leq 1}\sup_{\Vert f_{2}\Vert_{\infty},\dots,\Vert f_{d}\Vert_{\infty}\leq 1}\norm{\mathbb{E}_{0\leq r\leq L(N)}c_{r}\prod_{i=1}^{d}T_{i}^{[p_{N}(r)]}f_{i}}^{\kappa}_{2},
		\end{equation*}
		finishing the proof.
		\end{proof}

 We will demonstrate how we use the previous approach for a Hardy field function \[h(x)=s_h(x)+p_h(x).\] We do this with two specific examples to cover both cases, i.e., $d_{p_h}<d_{s_h}+1,$ and $d_{p_h}\geq d_{s_h}+1$, where $d_{p_h}$ is the degree of $p_{h}$ and $d_{s_h}$ is the degree of $s_{h}$ (for the general case, see right after the proof of Theorem~\ref{T:upper bound}).  In both cases, we will choose a positive integer $K$ that will indicate the order of the Taylor expansion for the non-polynomial part.

\begin{example}\label{E:1} Let $h_1(x)=\pi x+x\log x.$ 

Here we have $p_{h_1}(x)=\pi x,$ and $s_{h_1}(x)=x\log x,$ hence 
$d_{p_{h_1}}=1<1+d_{s_{h_1}}.$ 

\noindent Picking $K=d_{s_{h_1}}+1=2,$ we have that $h_1(N+r)$ is approximated by
\begin{eqnarray*}
p_{N,1}(r) & = & p_{h_1}(N+r)+s_{h_1}(N)+s'_{h_1}(N)r+\frac{s''_{h_1}(N)}{2}r^2 \\ & = & \pi N+N\log N+(\pi+1+\log N)r+\frac{1}{2N}r^2,
\end{eqnarray*}
$0\leq r\leq L_1(N)=N^{\frac{7}{12}}$ (we picked $L_1(x)$ as the geometric mean of $|s''_{h_1}(x)|^{-\frac{1}{2}}$ and $|s'''_{h_1}(x)|^{-\frac{1}{3}}$).
\end{example}

\begin{example}\label{E:2} Let $h_2(x)=\sqrt{2}x^2+\log^2 x.$

Here we have $p_{h_2}(x)=\sqrt{2}x^2,$ and $s_{h_2}(x)=\log^2 x,$ hence 
$d_{p_{h_2}}=2> 1+d_{s_{h_2}}.$

\noindent Picking $K=d_{p_{h_1}}+1=3,$ we have that $h_2(N+r)$ is approximated by
\begin{eqnarray*}
p_{N,2}(r) & = & p_{h_2}(N+r)+s_{h_2}(N)+s'_{h_2}(N)r+\frac{s''_{h_2}(N)}{2}r^2+ \frac{s'''_{h_2}(N)}{6}r^3 \\ & = & \sqrt{2}N^2+\log^2 N+\Big(2\sqrt{2}N+\frac{2\log N}{N}\Big)r+\Big(\sqrt{2}+\frac{1-\log N}{N^2}\Big)r^2+\frac{2\log N-3}{3N^3}r^3,
\end{eqnarray*}
$0\leq r\leq L_2(N)=N/(\log^{7/24}N)$ (here we picked $L_2(x)$ to be of the same growth rate as the geometric mean of $|s'''_{h_2}(x)|^{-\frac{1}{3}}$ and $|s^{(4)}_{h_2}(x)|^{-\frac{1}{4}}$).
\end{example}
  
As we already mentioned,   and it is verified by both examples we just saw, 
we will deal with functions such that the corresponding variable sequence $(p_N)_N$ doesn't have leading coefficient $1.$ In that case, we will transform it into such.  This is crucial to our study in order to use the concatenation approach from \cite{dfks}.  The following proposition justifies this and can be viewed as a change-of-variables procedure.
		
\begin{proposition}\label{P:cov}
Let $(a_{N})_N$ be a sequence of real numbers  with (eventually) constant sign, $L$ a positive function with $1\prec L(x)\prec x$, and $K\in\mathbb{N}$ such that
\begin{itemize}
    \item $\lim_{N\to\infty} L(N)|a_{N}|^{\frac{1}{K}}=\infty$;
    \item $\lim_{N\to\infty}a_{N}= 0$; and
    \item  $L(N) \ll |a_{N}|^{-\frac{K+1}{K^{2}}}$.
\end{itemize}
If $(p_{N})_N$ is a variable polynomial sequence 
of degree less than $K$, then there exist a variable polynomial sequence $(\tilde{p}_{N})_N$ of degree less than $K$ and a positive function $\tilde{L}$ with $1\prec \tilde{L}(x)\prec x$ such that for every system $(X,\mathcal{B},\mu,T_{1},\dots,T_{d})$, $f_1\in L^{\infty}(\mu)$, and $\kappa\in\N$, we have
	\begin{equation}\label{cvv_2}
	\begin{split}
			&\limsup_{N \to \infty}\sup_{\vert c_{n}\vert\leq 1}\sup_{\Vert f_{2}\Vert_{\infty},\dots,\Vert f_{d}\Vert_{\infty}\leq 1}\norm{\mathbb{E}_{0\leq n\leq L(N)}c_{n}\prod_{i=1}^{d}T_{i}^{[a_{N}n^{K}+p_{N}(n)]}f_{i}}^{\kappa}_{2}
			\\& \leq \limsup_{N \to \infty}\sup_{\vert c_{n}\vert\leq 1}\sup_{\Vert f_{2}\Vert_{\infty},\dots,\Vert f_{d}\Vert_{\infty}\leq 1}\norm{\mathbb{E}_{0\leq n\leq \tilde{L}(N)}c_{n}\prod_{i=1}^{d}T_{i}^{[n^{K}+\tilde{p}_{N}(n)]}f_{i}}^{\kappa}_{2}.
		\end{split}	
		\end{equation}
	\end{proposition}	

\begin{proof}
 For convenience denote $D_{N}:=|a_{N}|^{\frac{1}{K}}$. 
 We assume without loss of generality that, for large $N,$ $a_N>0.$ We have $\lim_{N\to\infty}D_{N}^{-1}=\infty$.
		For $0\leq n\leq L(N)$, we may write \[n=k[D_{N}^{-1}]+s,\] for some $0\leq k\leq [\tilde{L}(N)],$ where $\tilde{L}(N):=L(N)/[D_{N}^{-1}]$ 
  and $0\leq s\leq [D_{N}^{-1}]-1$.\footnote{ In case $L, a\in\mathcal{H},$ this $\Tilde{L}$ can actually be taken in $\mathcal{H}$ as well (in particular we can set $\Tilde{L}(N)$ to be equal to $L(N)D_N$) by the cost of an average that goes to $0$ as we lose values from a set of density $0$.}  Then
		$$a_Nn^K+p_{N}(n)=a_{N}(k[D_{N}^{-1}]+s)^{K}+p_{N}(k[D_{N}^{-1}]+s)=k^{K}(D_{N}[D_{N}^{-1}])^{K}+p_{N,s}(k)$$
		for some polynomial $p_{N,s}$ of degree at most $K-1$.
  
  Note that, if $N$ is large,  
		\[\vert k^{K}(D_{N}[D_{N}^{-1}])^{K}-k^{K}\vert\ll 1.\]
		Indeed,
		\begin{equation}\nonumber
		\begin{split}
		 \big\vert k^{K}(D_{N}[D_{N}^{-1}])^{K}-k^{K}\big\vert
	 	 &\leq Kk^{K}\big\vert D_{N}[D_{N}^{-1}]-1\big\vert 
		 =Kk^{K}\big\vert D_{N}\left(D_{N}^{-1}-\{D_{N}^{-1}\}\right)-1\big\vert 
		 \leq Kk^{K}D_{N}\\ &\leq KD_{N}\left(\frac{L(N)}{D_{N}^{-1}-1}\right)^{K}
		\ll KD_{N}\left(\frac{L(N)}{D_{N}^{-1}}\right)^{K}=K\left(\frac{L(N)}{D_{N}^{-\frac{K+1}{K}}}\right)^{K}\ll 1.
		\end{split}
		\end{equation}
So, 
\[a_N n^K+p_N(n)= k^K+p_{N,s}(k)+O(1).\]
 The left hand side of
 \eqref{cvv_2}, by using convexity, and then \cite[Lemma 3.2]{Ts} to deal with the bounded error terms, is bounded by a constant multiple of
\begin{equation*}
		\begin{split}
		&\quad\limsup_{N \to \infty}\mathbb{E}_{0\leq s\leq [D_{N}^{-1}]-1}\sup_{\vert c_{n}\vert\leq 1}\sup_{\Vert f_{2}\Vert_{\infty},\dots,\Vert f_{d}\Vert_{\infty}\leq 1}\norm{\mathbb{E}_{0\leq n\leq \tilde{L}(N)}c_{n}\prod_{i=1}^{d}T_{i}^{[n^{K}+p_{N,s}(n)]}f_{i}}^{\kappa}_{2}.
		\end{split}
		\end{equation*} 
		Since \[\lim_{N\to\infty}\tilde{L}(N)=\lim_{N\to\infty}L(N)D_{N}\frac{D_{N}^{-1}}{[D_{N}^{-1}]}=\lim_{N\to\infty}L(N)D_{N}=\lim_{N\to\infty}L(N)|a_{N}|^{\frac{1}{K}}=\infty,\] 
		and \[\lim_{N\to\infty} \frac{\tilde{L}(N)}{N}=\lim_{N\to\infty}\frac{L(N)}{N}D_{N}\frac{D_{N}^{-1}}{[D_{N}^{-1}]}=\lim_{N\to\infty}\frac{L(N)}{N}D_{N}=0,\]
		by setting $\tilde{p}_{N}$ to be the $p_{N,s}$ which attends the maximum of 
		\begin{equation*}
		\sup_{\vert c_{n}\vert\leq 1}\sup_{\Vert f_{2}\Vert_{\infty},\dots,\Vert f_{d}\Vert_{\infty}\leq 1}\norm{\mathbb{E}_{0\leq n\leq \tilde{L}(N)}c_{n}\prod_{i=1}^{d}T_{i}^{[n^{K}+p_{N,s}(n)]}f_{i}}^{\kappa}_{2},
		\end{equation*} we get the result.
	\end{proof}			

 In particular, for Example~\ref{E:1}, setting $a_N=\frac{s''_{h_1}(N)}{2}=\frac{1}{2N},$ we can pick $\Tilde{L}_1(N)=\frac{L_1(N)}{[a_N^{-1/2}]}=\frac{N^{\frac{7}{12}}}{[\sqrt{2N}]}$ (which grows as $\frac{N^{\frac{1}{12}}}{\sqrt{2}}$), so, for $n=k[a_N^{-1/2}]+s,$ $0\leq k\leq [\Tilde{L}_1(N)],$ $0\leq s\leq [a_N^{-1/2}]-1,$ we have 
\[p_{N,1}(n)= k^2+2a_N[a_N^{-1/2}]k +a_Ns^2+\tilde{p}_{N,1}(k[a_N^{-1/2}]+s)+O(1),\] 
where $\tilde{p}_{N,1}(r)=\pi N+N\log N+(\pi+1+\log N)r$ is of degree $1.$

\

Similarly, for Example~\ref{E:2}, setting $a_N=\frac{s'''_{h_2}(N)}{6}=\frac{2\log N-3}{3N^3},$ we can pick $\Tilde{L}_2(N)=\frac{L_2(N)}{[a_N^{-1/3}]}=\frac{N}{\log^{7/24}N}\cdot \frac{1}{\Big[\frac{\sqrt[3]{3}N}{\sqrt[3]{2\log N-3}}\Big]}$ (which grows as $\sqrt[3]{\frac{2}{3}}\log^{\frac{1}{24}}N$), so, for $n=k[a_N^{-1/3}]+s,$ $0\leq k\leq \Tilde{L}_2(N),$ $0\leq s\leq [a_N^{-1/3}]-1,$ we have 
\[p_{N,2}(n)= k^3+3a_N[a_N^{-1/3}]^2 sk^2+3a_N[a_N^{-1/3}]s^2 k+a_N s^3+\Tilde{p}_{N,2}(k[a_N^{-1/3}]+s)+O(1),\]
where $\Tilde{p}_{N,2}(r)=\sqrt{2}N^2+\log^2 N+\Big(2\sqrt{2}N+\frac{2\log N}{N}\Big)r+\Big(\sqrt{2}+\frac{1-\log N}{N^2}\Big)r^2$ is of degree $2.$

\section{PET induction for variable polynomials}\label{Sec_PET}
	
In this section we define the van der Corput operation, which will be used, together with the van der Corput lemma (Lemma \ref{ts43} and \ref{hkvdc}) and PET induction scheme, to get the required upper bounds of the expressions of interest. To achieve the latter, we also need to control the coefficients of the polynomial iterates (for which we follow \cite{dfks}).

\subsection{Van der Corput lemmas and van der Corput operation}	

We will use two different versions of the van der Corput lemma. 

\begin{lemma}[Lemma 4.3 of \cite{Ts}]\label{ts43}
    Let $(u_{n})_{n\in\mathbb{Z}}$ be a sequence in a Hilbert space with $\Vert u_{n}\Vert\leq 1$, $d\geq 1$ and $M,N\in\mathbb{N}$. Then    $$\Vert\mathbb{E}_{n\in[N]}u_{n}\Vert^{2^{d}}\ll_{d}\frac{1}{M}+\Big(\frac{M}{N}\Big)^{2^{d-1}}+\mathbb{E}_{-M\leq m\leq M}\vert\mathbb{E}_{n\in[N]}\langle u_{n+m},u_{n}\rangle\vert^{2^{d-1}}.$$
\end{lemma}

The next follows from from Chapter 21, Section 1.2, Lemma 1 of \cite{HK18}:

\begin{lemma}\label{hkvdc}  
    Let $(u_{n})_{n\in\mathbb{Z}}$ be a sequence in a Hilbert space with $\Vert u_{n}\Vert\leq 1$ and $M,N\in\mathbb{N}$. Then    $$\Vert\mathbb{E}_{n\in[N]}u_{n}\Vert^{2}\leq\frac{6M}{N}+\mathbb{E}_{x,y\in[M]}\mathbb{E}_{n\in[N]}\langle u_{n+x},u_{n+y}\rangle.$$
\end{lemma}
	
The PET induction is an inductive procedure to reduce the complexity of multiple ergodic averages, which was first introduced in \cite{wmpet}. In this paper, we use a variation of the PET induction scheme introduced in \cite{dfks}, adapted to the families of variable polynomials. 

We say that a sequence of polynomials $q=(q_{N})_{N\in\mathbb{N}},$ $q_{N}\colon\mathbb{Z}^{s}\to\mathbb{R}$ is \emph{consistent} if the degree of $q_{N}$ with respect to the first variable is, for $N$ sufficiently large, a constant. In this case this constant is defined to be the \emph{degree} of the sequence, denoted by $\deg(q)$. 
	
We say that a consistent sequence $q$ is \emph{essentially non-constant} if $\deg(q)>0$, and that two consistent sequences $q$ and $q'$ are \emph{essentially distinct} if $q-q'$ is essentially non-constant. We say that a tuple of polynomial sequences $(q_{1},\dots,q_{\ell})$
is \emph{consistent} if all of $q_{i}, q_{i}-q_{j},$ $i\neq j,$ are consistent, and \emph{non-degenerate} if the $q_i$'s are essentially non-constant and essentially distinct.

Let $s\in\mathbb{N}_{0}$ and $\ell\in\N$. For $1\leq m\leq \ell$ and $N\in\N$, let $q_{N,m}\colon\Z^{s+1}\to\R^{d}$ be a polynomial. 
 Put $\mathbf{q}=(q_{N,1},\dots,q_{N,\ell})_{N}$. We say that $A=(s,\ell,\mathbf{q})$ is a \emph{PET-tuple}.\footnote{We use $s$ instead of $s+1$ to highlight the number of $h_i$'s.} The tuple $A=(s,\ell,\mathbf{q})$ is \emph{non-degenerate}  (resp. \emph{consistent}) if 
	$\mathbf{q}$ is non-degenerate (resp. consistent).
	
\	
	
	For each non-degenerate PET-tuple $A=(s,\ell,\mathbf{q})$ and $1\leq t\leq \ell$, we define the \emph{vdC-operation}, 
	$\partial_{t}A$, according to the following three steps:
	
	\medskip
	
	{\bf Step 1}:  For all $1\leq m\leq \ell$ and $N\in\N$, let 
	$q_{N,1}',\ldots,q_{N,2\ell}': \Z^{s+2}\to\mathbb{R}^{d}$ be functions defined as 
	\[\displaystyle q'_{N,m}(n;h_1,\ldots,h_{s+1})=\left\{ \begin{array}{ll} q_{N,m-\ell}(n;h_1,\ldots,h_{s})-q_{N,t}(n;h_1,\ldots,h_{s}) & \; ,\ \ell +1\leq m\leq 2\ell\\ q_{N,m}(n+h_{s+1};h_1,\ldots,h_{s})-q_{N,t}(n;h_1,\ldots,h_{s})  & \; , \ 1\leq m\leq \ell\end{array} \right..\]
   We use the letter $n$ for the first variable and $h_i$'s for the remaining ones.
For convenience, we	write $q'_{m}:=(q'_{N,m})_{N}$ and let $q'_{0}$ denote the sequence of constant zero polynomials.
	
	\begin{lemma}\label{nono}
	If $\mathbf{q}$ is non-degenerate, then
for all $0\leq i, j\leq 2\ell, i\neq j$,  $q'_{i}-q'_{j}$ is consistent.
	\end{lemma}
	\begin{proof}
	Unpacking the definitions, it suffices to verify that the following families are consistent for all $0\leq i, j\leq 2\ell, i\neq j$:
	\begin{enumerate}[(i)]
	    \item $q_{N,i}(n;h_{1},\dots,h_{s})-q_{N,j}(n;h_{1},\dots,h_{s}), N\in\N$;
	     \item $q_{N,i}(n+h_{s+1};h_{1},\dots,h_{s})-q_{N,j}(n+h_{s+1};h_{1},\dots,h_{s}), N\in\N$;
	      \item $q_{N,i}(n+h_{s+1};h_{1},\dots,h_{s})-q_{N,j}(n;h_{1},\dots,h_{s}), N\in\N$.
	\end{enumerate}
	The first case follows from the assumption that $\mathbf{q}$ is non-degenerate. The second case follows from the assumption that $\mathbf{q}$ is non-degenerate and the fact that $q_{N,i}(n+h_{s+1};h_{1},\dots,h_{s})-q_{N,j}(n+h_{s+1};h_{1},\dots,h_{s})$ has the same leading coefficient in the variable $n$ as that of $q_{N,i}(n;h_{1},\dots,h_{s})-q_{N,j}(n;h_{1},\dots,h_{s})$. The third case is similar to the second one. 
	\end{proof}
	
	{\bf Step 2}: 
	We remove from 
 $q'_{1},\dots,q'_{2\ell}$
 the collections of functions 
 $q'_{j}$
 which are essentially constant and the corresponding functions with those as iterates, 
	and then put the  remaining ones into groups $J_{i}=\{(q''_{N,i,1})_{N},\dots,$ $(q''_{N,i,t_{i}})_{N}\},$ $ 1\leq i\leq r,$ for some $r,$ $t_{i}\in\mathbb{N}$ such that two sequences 
	are essentially distinct 
	if, and only if, they belong to different groups. 
	For every $1\leq j\leq t_i,$ there exist  variable polynomials $p''_{N,i,j}\colon\Z^{s+1}\to\R$ such that 
	$q''_{N,i,j}(n;h_{1},\dots,h_{s+1})= q''_{N,i,1}(n;h_{1},\dots,h_{s+1})+p''_{N,i,j}(h_{1},\dots,h_{s+1})$ for sufficiently large $N$. 
	
	\medskip
	
	{\bf Step 3}:  
	Let $q^{\ast}_{N,i}=q''_{N,i,1}$. 
	Set $\mathbf{q}^{\ast}=(q^{\ast}_{N,1},\dots,q^{\ast}_{N,r})_{N\in\N}$,  
	and let this new PET-tuple be $\partial_{t}A=(s+1,r,\mathbf{q}^{\ast})$.\footnote{ Here we abuse the notation by writing  $\partial_{t}A$ to denote any of such operations obtained from Step 1 to 3. Strictly speaking, $\partial_{t}A$ is not uniquely defined as the order of grouping of $q'_{N,1},\dots,q'_{N,2\ell}$ in Step 2 is ambiguous. However, this is done without loss of generality, since the order does not affect the value of $S(\partial_{t}A, \cdot)$ (see below).}  
	It is clear from the construction that $\mathbf{q}^{\ast}$ and $\partial_{t}A$ are non-degenerate. Therefore, if $A$ is non-degenerate, then so is $\partial_{t}A$.
	
	\medskip

	We say that the operation $A\to\partial_{t}A$ is \emph{1-inherited} if $q'_{1}=q^{\ast}_{1}$ and we did not drop $q^{\ast}_{1}$ or group it with any other $q^{\ast}_{i}$ in Step 2.

		Let
		$A=(s,\ell,\mathbf{q})$ be a PET-tuple, where  $\q=(q_{N,1},\ldots, q_{N,\ell})_{N}$ with $q_{N,i}\colon \Z^{s+1} \to \R^d$ being polynomials, $\kappa\in\mathbb{N}$, $(X,\mathcal{B},\mu,(T_n)_{n\in\mathbb{Z}^{d}})$ be a $\Z^{d}$-system, and $f\in L^{\infty}(\mu)$. 
		For $h_1,\ldots,h_s\in \Z$, set 
		\[S(A,f,\kappa,(h_1,\ldots,h_s))\coloneqq \overline{\lim}_{N\to\infty}\sup_{\vert c_{n}\vert\leq 1}\sup_{\Vert g_{2}\Vert_{\infty},\dots,\Vert g_{\ell}\Vert_{\infty}\leq 1}\Bigl\Vert\mathbb{E}_{n\in [N]}c_{n}\prod_{m=1}^{\ell}T_{{[q_{N,m}(n;h_{1},\dots,h_{s})]}}g_{m}(x)\Bigr\Vert_{2}^{\kappa},\]
		where $g_{1}:=f$, 
		and 		
		\begin{align*} S(A,f,\kappa)& \coloneqq \Es_{h_{1},\dots,h_{s}\in\Z}\overline{\lim}_{N\to\infty}\sup_{\vert c_{n}\vert\leq 1}\sup_{\Vert g_{2}\Vert_{\infty},\dots,\Vert g_{\ell}\Vert_{\infty}\leq 1}\norm{\mathbb{E}_{n\in [N]}c_{n}\prod_{m=1}^{\ell}T_{{[q_{N,m}(n;h_{1},\dots,h_{s})]}}g_{m}(x)}_{2}^{\kappa} \\ 
		&= \Es_{h_{1},\dots,h_{s}\in\Z} S(A,f,\kappa,(h_1,\ldots,h_s)).
		\end{align*}

\begin{lemma}\label{1234}
			Let $(X,\mathcal{B},\mu,(T_n)_{n\in\mathbb{Z}^{d}})$ be a $\Z^{d}$-system, $A=(s,\ell,\mathbf{q})$ be a non-degenerate PET-tuple, $f\in L^{\infty}(\mu),$ and  $\kappa\in\mathbb{N}$. Then, for any $1\leq t\leq \ell$, $\partial_{t}A$ is also a non-degenerate PET-tuple. Moreover, if $A\to\partial_{t}A$ is 1-inherited, then
			\[S(A,f,2\kappa)\ll_{\kappa,\ell} S(\partial_{t}A,f,\kappa).\]
		\end{lemma}
		\begin{proof}
			The fact that $\partial_{t}A$ is a non-degenerate  PET-tuple  was verified previously (see Step 3). 
			We are left with proving the second conclusion. For convenience, write  $\mathbf{h}\coloneqq (h_{1},\dots,h_{s})$ and  $\mathbf{h}'\coloneqq (h_{1},\dots,h_{s+1}).$ Suppose that $\partial_{t}A=(s+1,r,\mathbf{q}^{\ast})$.

Fix $\mathbf{h}=(h_{1},\dots,h_{s}).$ For every $N\in \mathbb{N},$ $1\leq n\leq N,$ we pick $|c_{N,n}|\leq 1,$ and $g_{N,m}\in L^\infty(\mu)$ with $\norm{g_{N,m}}_\infty\leq 1,$ $2\leq m\leq \ell,$ so that 
\[\norm{\mathbb{E}_{n\in [N]}c_{N,n}\prod_{m=1}^{\ell}T_{{[q_{N,m}(n;h_{1},\dots,h_{s})]}}g_{N,m}(x)}_{2}^{2\kappa}\]
is $1/N$ close to 
\[\sup_{\vert c_{n}\vert\leq 1}\sup_{\Vert g_{2}\Vert_{\infty},\dots,\Vert g_{\ell}\Vert_{\infty}\leq 1}\norm{\mathbb{E}_{n\in [N]}c_{n}\prod_{m=1}^{\ell}T_{{[q_{N,m}(n;h_{1},\dots,h_{s})]}}g_{m}(x)}_{2}^{2\kappa},\]
		where $g_{N,1}:=g_{1}:=f$.		
For $M, N\in \N,$ by Lemma \ref{ts43}, we have that  
\begin{equation} \label{bige}
\begin{split}
& \norm{\mathbb{E}_{n\in [N]}c_{N,n}\prod_{m=1}^{\ell}T_{[q_{N,m}(n;\h)]}g_{N,m}}^{2\kappa}_{2} \\
& \ll_{\kappa}\mathbb{E}_{\vert h_{s+1}\vert\leq M} \Big|\mathbb{E}_{n\in [N]}\Big \langle c_{N,n}\prod_{m=1}^{\ell}T_{[q_{N,m}(n;\h)]}g_{N,m},  c_{N,n+h_{s+1}}\prod_{m=1}^{\ell}T_{[q_{N,m}(n+h_{s+1};\h)]}g_{N,m} \Big \rangle \Big|^{\kappa} \\ & \hspace{13cm} +\frac{1}{M}+\Big(\frac{M}{N}\Big)^{\kappa} \\
& =\mathbb{E}_{\vert h_{s+1}\vert\leq M} 
\Big|\mathbb{E}_{n\in [N]}\Big \langle c_{N,n}\prod_{m=1}^{\ell}T_{[q_{N,m}(n;\h)]-[q_{N,t}(n;\h)]}g_{N,m}, \\ & \hspace{5cm}
			c_{N,n+h_{s+1}}\prod_{m=1}^{\ell}T_{[q_{N,m}(n+h_{s+1};\h)]-[q_{N,t}(n;\h)]}g_{N,m} \Big \rangle \Big|^{\kappa}   + \frac{1}{M}+\Big(\frac{M}{N}\Big)^{\kappa}	   \\
			& =\mathbb{E}_{\vert h_{s+1}\vert\leq M} 
			\Bigl|\mathbb{E}_{n\in [N]}\Big \langle c_{N,n}\prod_{m=1}^{\ell}T_{[q'_{N,m+\ell}(n;\h')]+\e_{N,m+\ell,n,\h'}}g_{N,m}, \\  & \hspace{5.8cm}
			c_{N,n+h_{s+1}}\prod_{m=1}^{\ell}T_{[q'_{N,m}(n;\h')]+\e_{N,m,n,\h'}}g_{N,m} \Big \rangle \Bigr|^{\kappa}   + \frac{1}{M}+\Big(\frac{M}{N}\Big)^{\kappa} \\
	& =\mathbb{E}_{\vert h_{s+1}\vert\leq M} 
			\Bigl|\mathbb{E}_{n\in [N]}\Big \langle c_{N,n}\overline{c}_{N,n+h_{s+1}},  
			 \prod_{m=1}^{2\ell}T_{[q'_{N,m}(n;\h')]+\e_{N,m,n,\h'}}g_{N,m} \Big \rangle \Bigr|^{\kappa}   + \frac{1}{M}+\Big(\frac{M}{N}\Big)^{\kappa},		
			\end{split}
			\end{equation}
			where  
   $\epsilon_{N,m,n,\h'}\in\{-1,0,1\}$ and $g_{m+\ell,N}:=\overline{g}_{m,N}$ for $1\leq m\leq \ell$.

Assume that $\q^{\ast}=(q^{\ast}_{N,1},\dots,q^{\ast}_{N,r})_{N}$. Since $A\to\partial_{t}A$ is 1-inherited, $q^{\ast}_{N,1}=q'_{N,1}$ and we did not drop or combine $(q'_{N,1})_{N}$ with any other polynomial. 
For $1\leq j\leq r$, let $I_{j}$ be the set of $m\in\{1,\dots,2\ell\}$ such that $(q'_{N,m})_{N}$ is essentially the same as $(q^{\ast}_{N,j})_{N}$. For $m\in I_{j}$, we may write 	$q'_{N,m}:=q^{\ast}_{N,j}+\tilde{q}_{N,j,m}$	
 for some variable polynomial family $(\tilde{q}_{N,j,m})_{N}$ which does not depend on the variable $n$ (i.e., the first variable) when $N$ is sufficiently large. Let $I_{0}$ be the set of $m\in\{1,\dots,2\ell\}$ such that $(q'_{N,m})_{N}$ is essentially constant.
 Then the last line of (\ref{bige}) is bounded by
 \begin{equation} \label{bige2}
\begin{split}
& \mathbb{E}_{\vert h_{s+1}\vert\leq M} 
			\Bigl|\mathbb{E}_{n\in [N]}\Big \langle \prod_{m\in I_{0}}T_{[q'_{N,m}(n;\h')]+\e_{N,m,n,\h'}}\overline{g}_{N,m},  
			\\& \quad\quad\quad\quad\quad\quad\quad\quad\quad\quad		 \overline{c}_{N,n}c_{N,n+h_{s+1}}\prod_{j=1}^{r}\prod_{m\in I_{j}}T_{[q'_{N,m}(n;\h')]+\e_{N,m,n,\h'}}g_{N,m} \Big \rangle \Bigr|^{\kappa}   + \frac{1}{M}+\Big(\frac{M}{N}\Big)^{\kappa}.			
\end{split}
\end{equation}
Since for all $m\in I_{0}$, $q'_{N,m}(n;\h')$ is independent of $n$ when $N$ is sufficiently large, we may write $q'_{N,m}(n;\h')=q''_{N,m}(\h')$ for some polynomial $q''_{N,m}$.
For any $\e=(\e_{m})_{m\in I_{0}}\in\{-1,0,1\}^{\vert I_{0}\vert}$, let $A_{N,\h',\e}$ denote the set of $n\in[N]$ such that $\epsilon_{N,m,n,\h'}=\e_{m}$ for all $m\in I_{0}$. Then
  we may rewrite (\ref{bige2}) as 
\begin{equation} \label{bige20}
\begin{split}
& \mathbb{E}_{\vert h_{s+1}\vert\leq M} 
			\Bigl|\mathbb{E}_{n\in [N]}\sum_{\e\in \{-1,0,1\}^{\vert I_{0}\vert}}\Big \langle \bold{1}_{A_{N,\h',\e}}(n)\prod_{m\in I_{0}}T_{[q''_{N,m}(\h')]+\e_{m}}\overline{g}_{N,m},  
			\\& \quad\quad\quad\quad\quad\quad\quad\quad\quad\quad		 \overline{c}_{N,n}c_{N,n+h_{s+1}}\prod_{j=1}^{r}\prod_{m\in I_{j}}T_{[q'_{N,m}(n;\h')]+\e_{N,m,n,\h'}}g_{N,m} \Big \rangle \Bigr|^{\kappa}   + \frac{1}{M}+\Big(\frac{M}{N}\Big)^{\kappa}
   \\&\ll_{\ell}\mathbb{E}_{\vert h_{s+1}\vert\leq M}\sup_{\e\in \{-1,0,1\}^{\vert I_{0}\vert}}\sup_{\vert c_{n}\vert\leq 1} 
			\Bigl|\mathbb{E}_{n\in [N]}\Big \langle \prod_{m\in I_{0}}T_{[q''_{N,m}(\h')]+\e_{m}}\overline{g}_{N,m},  
			\\& \quad\quad\quad\quad\quad\quad\quad\quad\quad\quad		 c_{n}\prod_{j=1}^{r}\prod_{m\in I_{j}}T_{[q'_{N,m}(n;\h')]+\e_{N,m,n,\h'}}g_{N,m} \Big \rangle \Bigr|^{\kappa}   + \frac{1}{M}+\Big(\frac{M}{N}\Big)^{\kappa}.
\end{split}
\end{equation}
When $N$ is sufficiently large, we may use the Cauchy-Schwarz inequality to bound the last line of
 (\ref{bige20}) by
 \begin{equation} \label{bige3}
\begin{split}
& \mathbb{E}_{\vert h_{s+1}\vert\leq M} 
			\sup_{\vert c_{n}\vert\leq 1}\norm{\mathbb{E}_{n\in [N]}c_{n}\prod_{j=1}^{r}\prod_{m\in I_{j}}T_{[q'_{N,m}(n;\h')]+\e_{N,m,n,\h'}}g_{N,m}}_2^{\kappa} + \frac{1}{M}+\Big(\frac{M}{N}\Big)^{\kappa} \\
& =\mathbb{E}_{\vert h_{s+1}\vert\leq M} \sup_{\vert c_{n}\vert\leq 1}
			\norm{\mathbb{E}_{n\in [N]}c_{n}\prod_{j=1}^{r}\prod_{m\in I_{j}}T_{[q^{\ast}_{N,j}(n;\h')]+[\tilde{q}_{N,j,m}(n;\h')]+\e'_{N,m,n,\h'}}g_{N,m}}_2^{\kappa} 
   + \frac{1}{M}+\Big(\frac{M}{N}\Big)^{\kappa},		
\end{split}
\end{equation}
where $\epsilon'_{N,m,n,\h'}\in\{-2,0,2\}$.
  Since $\partial_t A$ is 1-inherited we have that
  \[\prod_{m\in I_{1}}T_{[q^{\ast}_{N,m}(n;\h')]+[\tilde{q}_{N,j,m}(n;\h')]+\e'_{N,m,n,\h'}}g_{N,m}=T_{[q^{\ast}_{N,1}(n;\h')]+\e'_{N,1,n,\h'}}f.\]
Using \cite[Lemma~3.2]{Ts},  the last line of (\ref{bige3}) is bounded by $O_{\kappa,\ell}(1)$ times
\begin{equation}\nonumber
\begin{split}
& \mathbb{E}_{\vert h_{s+1}\vert\leq M} 
			\sup_{\vert c_{n}\vert\leq 1}\sup_{\Vert g^{\ast}_{m}\Vert_{\infty}\leq 1, \atop m\in \bigcup_{j=2}^{r}I_{j}}\norm{\mathbb{E}_{n\in [N]}c_{n}\prod_{j=1}^{r}\prod_{m\in I_{j}}T_{[q^{\ast}_{N,j}(n;\h')]+[\tilde{q}_{N,j,m}(n;\h')]}g^{\ast}_{m}}_2^{\kappa}   + \frac{1}{M}+\Big(\frac{M}{N}\Big)^{\kappa}
			\\  & \leq \mathbb{E}_{\vert h_{s+1}\vert\leq M} 
			\sup_{\vert c_{n}\vert\leq 1}\sup_{\Vert g_{2}\Vert_{\infty},\dots,\Vert g_{r}\Vert_{\infty}\leq 1}\norm{\mathbb{E}_{n\in [N]}c_{n}\prod_{j=1}^{r} T_{[q^{\ast}_{N,j}(n;\h')]}g_{j}}_2^{\kappa}   + \frac{1}{M}+\Big(\frac{M}{N}\Big)^{\kappa},
\end{split}
\end{equation}
where $g^{\ast}_{1}=f$. 
Taking the limsup as $N$ goes to infinity, we conclude that 
			\[ \limsup_{N \to \infty}\sup_{\vert c_{n}\vert\leq 1}\sup_{\Vert g_{2}\Vert_{\infty},\dots,\Vert g_{\ell}\Vert_{\infty}\leq 1} \norm{\mathbb{E}_{n\in [N]}c_{n}\prod_{m=1}^{\ell}T_{[q_{N,m}(n;\h)]}g_{m}}^{2\kappa}_{2} \leq  \mathbb{E}_{\vert h_{s+1}\vert\leq M}S(\partial_t A, \kappa, \h') +\frac{1}{M}.\]
			
			Letting $M$ go to infinity we get 
			\[ S(A,f,2\kappa,\h) \ll_{\kappa,\ell} \mathbb{E}_{h_{s+1}\in \Z} S(\partial_t A,f, \kappa, \h').   \]
			Taking the average $\mathbb{E}_{h_1,\ldots,h_s\in \Z}$, we obtain the desired conclusion.  
		\end{proof}

			Let $A=(s,\ell,\q=(q_{N,1},\dots,q_{N,\ell})_{N})$ be a PET tuple. Let $\deg(A)$, the \emph{degree} of $A$, be the maximum of $\deg((q_{N,j})_{N}), 1\leq j\leq \ell$.
 We say that  $A$ is \emph{1-standard} if  $\deg((q_{N,1})_{N})=\deg(A)$.

\begin{lemma}\label{reduction}
Let $A$ be a 1-standard and non-degenerate PET-tuple with $\deg(A)\geq 1$. There exist $M\in\N$ depending only on $\deg(A)$, $\ell \in\N$, and $i_{1},\dots,i_{M}\in\N$ such that for all $1\leq M'\leq M$, $\partial_{i_{M'-1}}\dots \partial_{i_{1}}A\to\partial_{i_{M'}}\dots \partial_{i_{1}}A$ is 1-inherited,\footnote{$\partial_{i_{K}}\dots \partial_{i_{1}}A$ is understood as $A$ when $K=0$} $\partial_{i_{M'}}\dots \partial_{i_{1}}A$ is 1-standard, non-degenerate, and that  $\deg(\partial_{i_{M}}\dots \partial_{i_{1}}A)=1$.\footnote{ If $\deg(A)=1$, then one can take $M=0$, and the claim is trivial.}
\end{lemma}

	\begin{proof}
		The proof is routine and almost identical to  \cite[Theorem~4.2]{DKS}; the additional requirements ``that for all $1\leq M'\leq M$, $\partial_{i_{M'-1}}\dots \partial_{i_{1}}A\to\partial_{i_{M'}}\dots \partial_{i_{1}}A$ is 1-inherited, $\partial_{i_{M'}}\dots \partial_{i_{1}}A$ is a non-degenerate and 1-standard'' follows directly from the proof  (see also \cite[Theorem~4.6]{dfks}, and its footnote).		
	\end{proof}

\subsection{Coefficient tracking}\label{s:ct}

While Lemma \ref{reduction} asserts that one can always transform a PET-tuple $A$  into a new PET-tuple $\partial_{i_{M'}}\dots \partial_{i_{1}}A$ of degree 1 using the vdC operations, it provides no information on the relation between the coefficients of the polynomials in  $\partial_{i_{M'}}\dots \partial_{i_{1}}A$ and that in the original PET-tuple $A$. Such information  will be essential in computing the upper bound of $S(A,f,\kappa)$. To overcome this difficulty, in \cite{dfks,DKS}, we introduced a machinery to keep track of the coefficients of relevant polynomials. In this paper, we adopt an approach similar to \cite{dfks} to control the coefficients. This section generalizes the results in \cite[Section 5]{dfks} to variable polynomials.

 Let $\p=(p_{N,1},\dots,p_{N,k})_{N}$ denotes a non-degenerate family of vectors of variable polynomials of degree at most $K$. Write $$p_{N,i}(n)=\sum_{v=0}^{K}b_{N,i,v}n^{v},$$ where $b_{N,i,v}\in \mathbb{R}^{d}$.
 For $N\in\mathbb{N},r\in\mathbb{Q}, v\in\N_0$ and   $0\leq i\leq k$, we set $0\leq i,j\leq k$, $i\neq j$, we set
 $$Q_{N,r,i,v}(\p)\coloneqq \{r(b_{N,w,v}-b_{N,i,v})\colon 0\leq w\leq k\}.$$

  For $0\leq i,j\leq k$, $i\neq j$, let $v_{i,j}$ be the largest integer such that $b_{N,i,v_{i,j}}\neq b_{N,j,v_{i,j}}$ and set 
  \[G'_{N,i,j}(\p):=\text{span}_{\mathbb{Q}}\{b_{N,i,v_{i,j}}-b_{N,j,v_{i,j}}\}.\]

 Let $\q=(q_{N,1},\dots,q_{N,\ell})_{N}$ denote a family variable polynomials, with \[q_{N,i}(n;h_{1},\dots,h_{s})=\sum_{b,a_{1},\dots,a_{s}\in\mathbb{N}_{0},b+a_{1}+\dots+a_{s}\leq K}u_{N,i}(b,a_{1},\dots,a_{s})n^{b}h_{1}^{a_{1}}\dots h_{s}^{a_{s}}\]
 for some $K\in\mathbb{N}_{0}$ and some $u_{N,i}(b,a_{1},\dots,a_{s})\in\mathbb{R}^{d}$.
 For $b,a_{1},\dots,a_{s}\in\mathbb{N}_{0}$, let
 $$u_{N}(\q,b;a_{1},\dots,a_{s}):=(u_{N,1}(b,a_{1},\dots,a_{s}),\dots,u_{N,\ell}(b,a_{1},\dots,a_{s}))$$
 and $\u(\q,b;a_{1},\dots,a_{s}):=(u_{N}(\q,b;a_{1},\dots,a_{s}))_{N}$.

 \begin{definition*}[Types and symbols of level data]
 For all $v\in\mathbb{N}_{0},r\in\mathbb{Q}$, and $0\leq i\leq k$, we say that a sequence of $\ell$-tuples $\u=(u_{N,1},\dots,u_{N,\ell})_{N}$, $u_{N,i}\in\mathbb{R}$ is of \emph{type $\p(r,i,v)$} if $$u_{N,1},\dots,u_{N,\ell}\in Q_{N,r,i,v}(\p),\;\text{and}\;\; u_{N,1}=r(b_{N,1,v}-b_{N,i,v})$$
 when $N$ is sufficiently large.
 
 Let $\u=(u_{N,1},\dots,u_{N,\ell})_{N}$ be of type $\p(r,i,v)$. Suppose that \[ (u_{N,1},\dots,u_{N,\ell})=
 	(r(b_{N,w_{1},v}-b_{N,i,v}),\dots,r(b_{N,w_{\ell},v}-b_{N,i,v})) ,\] for some $0\leq w_{1},\dots,w_{\ell}\leq k$ for all $N$ sufficiently large. We call $\w:=(w_{1},\dots,w_{\ell})$ a \emph{symbol} of $\u$. (Note that we always have $w_{1}=1$.)
 \end{definition*}

\begin{definition*} \label{def:P1--P4}
 Let $S$ denote the set of all $(a,a')\in\N_0^{2}$  such that $a$ and $a'$ are both $0$ or both different than $0$. Let $\p,\q$ be polynomial families of degree at least 1. We say that $\q$ satisfies (P1)--(P4) with respect to $\p$ if its level data   $\u(\q,\ast)$ satisfy: 

 (P1)  For all $a_{1},\dots,a_{s},b\in\N$, there exist $r\in\mathbb{Q},0\leq i\leq k,v\in\mathbb{N}_{0}$ such that 
 $\u(\q,b;a_{1},\dots,a_{s})$ is of type $\p(r,i,v)$.  Moreover, we may choose the type and symbol for all of $\u(\q,b;a_{1},\dots,a_{s})$ in a way such that (P2)--(P4) hold, where:

 (P2) Suppose that $\u(\q,b;a_{1},\dots,a_{s})$ is of type $\p(r,i,v)$, then  $r=\binom{b+a_{1}+\dots+a_{s}}{b,a_{1},\dots,a_{s}}$  and $v=b+a_{1}+\dots+a_{s}$ (in particular, $r\neq 0$).\footnote{ Here $\binom{b+a_{1}+\dots+a_{s}}{b,a_{1},\dots,a_{s}}:=\frac{(b+a_{1}+\dots+a_{s})!}{b!a_{1}!\dots a_{s}!}$.} 
 
 (P3) Suppose that $\u(\q,b;a_{1},\dots,a_{s})$ is of type $\p(r,i,v)$ and  $\u(\q,b';a'_{1},\dots,a'_{s})$ is of type $\p(r',i',v')$. If  $(a_{1},a'_{1}),\dots,(a_{s},a'_{s})\in S$, then $i=i'$  and $\u(\q,b;a_{1},\dots,a_{s}),$ $\u(\q,b';a'_{1},\dots,a'_{s})$ share a symbol $\bold{w}$.
 
 (P4) For any $\u(\q,b;a_{1},\dots,a_{s})$, the first coordinate $w_{1}$ of its symbol $(w_{1},\dots,w_{\ell})$ equals to 1.

 For convenience we say that a PET-tuple $A=(s,\ell,\q)$ satisfies (P1)--(P4) if the polynomial family $\q$ associated to $A$ satisfies (P1)--(P4).
 \end{definition*}

\begin{proposition}\label{PET2} 
 Let $A=(s,\ell,\q)$ be a non-degenerate PET-tuple and $1\leq \rho\leq \ell$. 
 	Assume that $A\to\partial_{\rho}A$ is 1-inherited.
 	If 
  $A$ satisfies (P1)--(P4), then 
  $\partial_{\rho}A$ also satisfies (P1)--(P4).
 \end{proposition}

\begin{proposition}\label{12-3}
 	Suppose that (P1)--(P4) hold for some non-degenerate $\q$ with respect to $\p$. Then for all $0\leq m\leq \ell, m\neq 1$ and $N$ sufficiently large, the group
 	 	\[H_{N,1,m}(\q)\coloneqq \text{\emph{span}}_{\mathbb{Q}}\big\{u_{N,1}(\q,b;a_{1},\dots,a_{s})-u_{N,m}(\q,b;a_{1},\dots,a_{s})\colon (b,a_{1},\dots,a_{s})\in\N_0^{s+1},b\neq 0\big\}\]
 contains at least one of the groups $G'_{N,1,j}(\p), 0\leq j\leq k, j\neq 1$.
 \end{proposition}

\begin{remark}
\emph{The proofs of	
Propositions \ref{PET2} and \ref{12-3} are almost identical to Propositions~5.6 and 5.7 of \cite{dfks}, for $L=1,$ modulo the following differences:
\begin{itemize}
    \item Propositions \ref{PET2} and \ref{12-3} are for variable polynomials while Propositions 5.6 and 5.7 of \cite{dfks} are about polynomials;
    \item the polynomials in Propositions \ref{PET2} and \ref{12-3} take values in $\mathbb{R}$ while the ones in Propositions 5.6 and 5.7 of \cite{dfks} take values in $\mathbb{Q}$;
    \item the groups $H_{N,1,m}(\q)$ and $G'_{N,1,j}(\p)$ defined and used in  Proposition  \ref{12-3} are different from the groups $H_{1,m}(\q)$ and $G_{1,j}(\p)$ defined and used in \cite[Proposition 5.7]{dfks} (we do not intersect these group with $\mathbb{Z}^{d}$ here).
\end{itemize}
One can easily check that the differences mentioned above do not affect the proofs in \cite{dfks}, and the same arguments can be used to prove Propositions \ref{PET2} and \ref{12-3} without difficulty. We leave the details to the interested readers.}
\end{remark}
	
\section{Bounding multiple ergodic averages with Host-Kra seminorms}\label{Sec_HK}

In this section we prove the Host-Kra-type bounds that we need for our main averages. More specifically, we prove Proposition~\ref{polytolinear2} which treats the basic, linear variable polynomial case, and Theorem~\ref{polytolinear}, which treats the case of variable polynomials of leading coefficient 1 (i.e., the ones that we are dealing with in our study).

We start with the definition of Host-Kra seminorms, which is a fundamental tool in studying problems related to multiple averages, and they were first introduced in \cite{HK99} for ergodic $\Z$-systems. A variation of these seminorms in the context of $\Z^{d}$-systems was introduced in \cite{hostcommuting}. As in \cite{DKS}, we will use a slightly more general version of these characteristic factors (see also \cite{Sun} for a similar approach).

For a $\Z^d$-system $(X,\mathcal{B},\mu,(T_{n})_{n\in \Z^d})$ and a subgroup $H$ of $\Z^d$, $\mathcal{I}(H)$ denotes the sub-$\sigma$-algebra of $(T_h)_{h\in H}$-invariant sets, i.e., sets $A\in \mathcal{B}$ such that $T_h A=A$ for all $h\in H$.  For an invariant sub-$\sigma$-algebra  $\mathcal{A}$ of $\mathcal{B}$, the measure $\mu\times_{\mathcal{A}} \mu$ denotes the \textit{relative independent product of $\mu$ with itself over $\mathcal{A}$}. That is, $\mu\times_{\mathcal{A}} \mu$ is the measure defined on the product space $X\times X$ as 
\[ \int_{X\times X} f\otimes g~ d(\mu\times_{\mathcal{A}} \mu)= \int_{X} \mathbb{E}(f\vert \mathcal{A})  \mathbb{E}(g \vert \mathcal{A})d\mu   \]
for all $f,g\in L^{\infty}(\mu)$. 

Let   $H_{1},\dots,H_{k}$ be subgroups of $\Z^d$. Define 
\[\mu_{H_{1}}=\mu\times_{\I(H_{1})}\mu\]
and for $k>1,$
\[\mu_{H_{1},\dots,H_{k}}=\mu_{H_{1},\dots,H_{k-1}}\times_{\mathcal{I}(H_{k}^{[k-1]})}\mu_{H_{1},\dots,H_{k-1}},\]
where $H^{[k-1]}_{k}$ denotes 
the subgroup of $(\Z^{d})^{2^{k-1}}$ consisting of all the elements of the form 
 $h_{k}\times\dots\times h_{k}$ ($2^{k-1}$ copies of $h_{k}$) for some $h_{k}\in H_{k}$. For $f\in L^{\infty}(\mu)$, its \emph{Host-Kra seminorm} $\nnorm{f}_{H_{1},\dots,H_{k}}$ is defined by
  \[\nnorm{f}_{H_{1},\dots,H_{k}}^{2^{k}}\coloneqq \int_{X^{[k]}} \prod_{\epsilon \in \{0,1\}^k}\mathcal{C}^{|\epsilon|}f\,d\mu_{H_{1},\dots,H_{k}},\] 
  where $X^{[k]}=X\times\cdots\times X$ ($2^k$ copies $X$), $|\epsilon|=\epsilon_1+ \ldots +\epsilon_k$ and $\mathcal{C}$ is the conjugation map $f\mapsto \overline{f}$.

For convenience, we adopt a flexible way to write the Host-Kra seminorms combining the aforementioned notation. For example, if $A=\{H_{1},H_{2}\}$, then the notation $\nnorm{\cdot}_{A,H_{3},H^{\times 2}_{4},(H_{i})_{i=5,6}}$ refers to  $\nnorm{\cdot}_{H_{1},H_{2},H_{3},H_{4},H_{4},H_{5},H_{6}}$.
For $g_{1},\dots,g_{t}\in \mathbb{Z}^{d}$, we denote  $\nnorm{\cdot}_{T_{g_{1}},\dots,T_{g_t}}$ as $\nnorm{\cdot}_{H_{1},\dots,H_{t}}$, where each $H_{i}$ is generated by $g_{i}$.

The following proposition, which has at the beginning an argument similar to that of \cite[Lemma~4.7]{Fra3}, allows us to bound weighted multiple ergodic averages with certain linear variable polynomial iterates uniformly by Host-Kra seminorms.   

\begin{proposition}\label{polytolinear2}  
Let $(X,\mathcal{B},\mu,(T_{n})_{n\in \Z^d})$ be a $\Z^d$-system, $\ell\in\mathbb{N}$ and $f_{1},\dots,f_{\ell}\in L^{\infty}(\mu)$ be bounded by 1. 
Let $k_{1},\dots,k_{\ell}\in\Z^{d}$ and let $(r_{N,m})_{N,m}$ be a sequence in $\Z^{d}$. If $\ell>1$, we have that 
\begin{equation}\label{E:base_case_k=1 g}
\limsup_{N\to\infty}\sup_{|c_{n}|\leq 1} \norm{\frac{1}{N}\sum_{n=1}^N c_{n} \prod_{m=1}^\ell T^{k_m n+r_{N,m}}f_m}_2\ll_\ell \nnorm{f_{1}}_{T^{k_{1}},T^{k_{1}},T^{k_{1}-k_{2}},\dots,T^{k_{1}-k_{\ell}}}.
\end{equation} 
Furthermore, if $\ell=1,$ then  the left hand side of \eqref{E:base_case_k=1 g} is bounded by $\norm{\mathbb{E}(f_{1}\otimes \overline{f_{1}} \vert I(T^{k_1}\times T^{k_1}))}_{L^{2}(\mu\times\mu)}^{1/2}.$  
\end{proposition}

\begin{remark}
\emph{For Proposition~\ref{polytolinear2} to be useful in the proof of Theorem~\ref{polytolinear}, it is crucial that the constant that appears in \eqref{E:base_case_k=1 g} depends only on the number of linear iterates. We highlight the fact that this would not be the case if we had, e.g., iterates of the form $[\alpha_m n],$ for vectors $\alpha_m$ of non-integer coordinates (the constant would then depend on $\alpha_1,\ldots,\alpha_\ell$ too).}
\end{remark}

We need the following lemma in the proof of Proposition \ref{polytolinear2}.

\begin{lemma}\label{nhn}
 For any sequence $f\colon\mathbb{Z}\to\mathbb{C}$ bounded by 1, if $\lim_{N\to\infty}\mathbb{E}_{n\in[-N,N]}f(n)$ exists, then
    \[\lim_{N\to\infty}\mathbb{E}_{n\in[-N,N]}\frac{2(N+1-\vert n\vert)}{N+1}f(n)=\lim_{N\to\infty}\mathbb{E}_{n\in[-N,N]}f(n).\]
\end{lemma}
\begin{proof}
We have the following relation
\begin{equation}\nonumber
\begin{split}
\quad\mathbb{E}_{n\in[-N,N]}\frac{2(N+1-\vert n\vert)}{N+1}f(n) &=\frac{2}{(2N+1)(N+1)}\sum_{n=-N}^{N}\sum_{M=|n|}^{N} f(n)   \\
&=\frac{2}{(2N+1)(N+1)}\sum_{M=0}^{N}\sum_{n=-M}^{M}f(n)
\\&=\frac{2}{(2N+1)(N+1)}\sum_{M=0}^{N}(2M+1)\mathbb{E}_{n\in[-M,M]}f(n).
\end{split}
\end{equation}
Noting that $\lim_{N\to\infty}\Big\vert\frac{2}{(2N+1)(N+1)}\sum_{M=0}^{N}(2M+1)-1\Big\vert=0$, the claim follows.
\end{proof}

\begin{proof}[Proof of Proposition \ref{polytolinear2}]
If $\ell=1,$ we let $B:=\norm{\mathbb{E}(f_{1}\otimes \overline{f_{1}} \vert I(T^{k_1}\times T^{k_1}))}_{L^{2}(\mu\times\mu)}^{1/2},$ while if $\ell>1$ we set $B:=\nnorm{f_{1}}_{T^{k_{1}},T^{k_{1}},T^{k_{1}-k_{2}},\dots,T^{k_{1}-k_{\ell}}}$.  
For every $N\in \mathbb{N},$ and $1\leq n\leq N,$ pick $|c_{N,n}|\leq 1,$ so that the corresponding norm  in the left hand side of \eqref{E:base_case_k=1 g} is $1/N$ close to its supremum $\sup_{|c_{n}|\leq 1}$. So, it suffices to show that
\begin{equation}\label{E:base_case_k=1 0}
\limsup_{N\to\infty} \norm{\frac{1}{N}\sum_{n=1}^N c_{N,n} \prod_{m=1}^\ell T^{k_m n+r_{N,m}}f_m }_2\ll_{\ell} B.
\end{equation}
To this end, it suffices to show 
\begin{equation}\label{E:base_case_k=1_2}
\limsup_{N\to\infty}\sup_{\norm{f_0}_\infty\leq 1}\frac{1}{N}\sum_{n=1}^N \left|\int f_0\cdot \prod_{m=1}^\ell T^{k_m n+r_{N,m}}f_m \;d\mu\right|\ll_{\ell} B^{2}.
\end{equation}
Indeed, assuming \eqref{E:base_case_k=1_2} and using the triangle inequality, whenever $f_{N,0}\in L^\infty(\mu)$ with $\norm{f_{N,0}}_\infty\leq 1$ for $N\in \mathbb{N},$ we have
\begin{equation}\label{E:base_case_k=1_3}
\limsup_{N\to\infty}\left\vert\frac{1}{N}\sum_{n=1}^N c_{N,n} \int f_{N,0}\cdot \prod_{m=1}^\ell T^{k_m n+r_{N,m}}f_m\right\vert \;d\mu
     \ll_\ell B^{2}.
\end{equation}
Using \eqref{E:base_case_k=1_3} with the conjugate of $\frac{1}{N}\sum_{n=1}^N c_{N,n} \prod_{m=1}^\ell T^{k_m n+r_{N,m}}f_m$ in place of $f_{N,0},$ we get \eqref{E:base_case_k=1 0}. 

We now prove (\ref{E:base_case_k=1_2}).
Given $N\in \N$ and $\|f_0\|_{\infty}\leq 1,$ we have  
	\begin{eqnarray*} \left(\frac{1}{N}\sum_{n=1}^N \left|\int f_0\cdot \prod_{m=1}^\ell T^{k_m n+r_{N,m}}f_m \;d\mu\right| \right)^2  &\leq & \frac{1}{N}\sum_{n=1}^N \left|\int f_0\cdot \prod_{m=1}^\ell T^{k_m n+r_{N,m}}f_m \;d\mu\right|^2\\ 
& = & \int F_{N,0} \cdot  \frac{1}{N}\sum_{n=1}^N  S^{k_1n} F_{1} \cdot  \prod_{m=2}^{\ell} S^{k_m n}F_{N,m}\; d(\mu\times \mu),
\end{eqnarray*}
where $S=T\times T$, $F_{N,0}=T^{-r_{N,1}}f_0\otimes T^{-r_{N,1}} \overline{f_0}$, $F_1=f_1 \otimes \overline{f_1}$ and $F_{N,m}= T^{r_{N,m}-r_{N,1}}f_m \otimes   T^{r_{N,m}-r_{N,1}}\overline{f_m}$.  Using the Cauchy-Schwarz inequality, we can bound the latter expression by 
\begin{equation}\label{E:base_case_k=1_2 g e=0 6}  \norm{\frac{1}{N}\sum_{n=1}^N S^{k_1 n}F_{1}\cdot \prod_{m=2}^\ell S^{k_mn}F_{N,m}}_{L^2(\mu\times \mu)}.  \end{equation}								
 Note that this bound is uniform for all $\|f_0\|_{\infty}\leq 1$. Consider the case $\ell=1$. Letting $N\to \infty$ and using the von Neumann Ergodic Theorem, we get that the limit of \eqref{E:base_case_k=1_2 g e=0 6} (for $\ell=1)$ as $N\to\infty$ can be bounded by
	\[  \| \mathbb{E}(F_{1} | I(S^{k_1}))\|_{2}= \Vert\mathbb{E}(f_{1}\otimes \overline{f_{1}} \vert I(T^{k_1}\times T^{k_1}))\Vert_{L^{2}(\mu\times\mu)}.\]
 Hence, we obtain the desired conclusion for $\ell=1$. 
	
We now consider the case $\ell\geq 2$. Let $0\leq t \leq \ell-1$, and denote $\h'\coloneqq (\h,h_{t+1}):=(h_{1},\dots,h_{t+1})$ and ${\boldsymbol{\epsilon}}'\coloneqq ({\boldsymbol{\epsilon}},\e_{t+1})\coloneqq (\e_{1},\dots,\e_{t+1})$. For  $a_{1},\dots,a_{\ell-t}\in\mathbb{Z}^{d}$ and $b_1,\ldots,b_t\in \Z^d$,  consider the quantity $\tilde{S}(t,\kappa,(a_i)_{i=1}^{\ell-t},(b_i)_{i=1}^t)$ defined as
\begin{equation}\label{eq_bound_cubic}
\E_{\bold{h}\in [-M,M]^{t}}\sup_{\Vert F_{2}\Vert_{\infty},\dots,\Vert F_{\ell-t}\Vert_{\infty}\leq 1}\norm{ \E_{n\in[N]}\Big(S^{a_{1}n}\prod_{\e\in\{0,1\}^{t}}\prod_{i=1}^{t}S^{b_{i}h_{i}\e_{i}}\Big)\mathcal{C}^{\vert \e\vert}F_{1}\cdot \prod_{m=2}^{\ell-t}S^{a_{m}n}F_{m}}_{2}^{\kappa}.
\end{equation} 

Note that $\tilde{S}(0,\kappa,(k_i)_{i=1}^{\ell},\emptyset)$\footnote{Here we adopt the natural convention that when $t=0$, $\prod_{\e\in\{0,1\}^{t}}\Big(\prod_{i=1}^{t}S^{b_{i}h_{i}\e_{i}}\Big)\mathcal{C}^{\vert \e\vert}$ is the identity map.} is a bound for \eqref{E:base_case_k=1_2 g e=0 6} to the power of $\kappa$. 

We claim that for all $0\leq t\leq \ell-1$, and $\kappa\in\mathbb{N}$, we have
\begin{equation} \label{eq_s_tilde} \tilde{S}(t,2\kappa,(a_i)_{i=1}^{\ell-t},(b_i)_{i=1}^t)\ll_{\ell,\kappa} \frac{1}{M} + \Bigl(\frac{M}{N}\Bigr)^{\kappa} + \tilde{S}(t+1,\kappa,(a'_i)_{i=1}^{\ell-t-1}, (b'_i)_{i=1}^{t+1}),\end{equation}

where $a_i'=a_i-a_{\ell-t}$ for $1\leq i\leq \ell -t-1,$ and $b_i'=b_i$ for $1\leq i \leq t$ and $b'_{t+1}=a_1$. 

Indeed, by Lemma \ref{ts43}, $\tilde{S}(t,2\kappa,(a_i)_{i=1}^{\ell-t},(b_i)_{i=1}^t)$ can be bounded by $O_{\kappa,\ell}(1)$ times
\begin{align*} 
\E_{\bold{h}\in [-M,M]^{t}}\sup_{\Vert F_{2}\Vert_{\infty},\dots,\Vert F_{\ell-t}\Vert_{\infty}\leq 1}\Bigl( &\E_{h_{s+1}\in[-M,M]}\Bigl\vert \E_{n\in[N]}\Big\langle \Big(S^{a_{1}n}\prod_{\e\in\{0,1\}^{t}}\prod_{i=1}^{t}S^{b_{i}h_{i}\e_{i}}\Big)\mathcal{C}^{\vert \e\vert}F_{1}\cdot \prod_{m=2}^{\ell-t}S^{a_{m}n}F_{m},
\\&\qquad\quad\Big(S^{a_{1}(n+h_{t+1})}\prod_{\e\in\{0,1\}^{t}}\prod_{i=1}^{t}S^{b_{i}h_{i}\e_{i}}\Big)\mathcal{C}^{\vert \e\vert}F_{1}\cdot \prod_{m=2}^{\ell-t}S^{a_{m}(n+h_{t+1})}F_{m}\Big\rangle\Bigr\vert^{\kappa}\\ &\quad\quad\quad\quad\quad + \frac{1}{M} + \Big(\frac{M}{N}\Big)^{\kappa}\Bigr).
\end{align*}
 Noting that 
\begin{align*}
&\Bigl(S^{a_{1}n}\prod_{\e\in\{0,1\}^{t}}\prod_{i=1}^{t}S^{b_{i}h_{i}\e_{i}}\Big)\mathcal{C}^{\vert \e\vert}F_{1} \cdot \Bigl(S^{a_{1}(n+h_{t+1})}\prod_{\e\in\{0,1\}^{t}}\prod_{i=1}^{t}S^{b_{i}h_{i}\e_{i}}\Big)\mathcal{C}^{\vert \e\vert} \overline{F_{1}} \\ &= S^{a_{1}n}\prod_{\e'\in\{0,1\}^{t+1}}\Big(\prod_{i=1}^{t+1}S^{b_{i}h_{i}\e_{i}}\Big)\mathcal{C}^{\vert \e'\vert}F_{1},
\end{align*} and using the invariance of the measure, the previous expression equals to 
\begin{align*}
&\E_{\bold{h}\in [-M,M]^{t}}\sup_{\Vert F_{2}\Vert_{\infty},\dots,\Vert F_{\ell-t}\Vert_{\infty}\leq 1}\Bigl(\frac{1}{M}+\Big(\frac{M}{N}\Big)^{\kappa}
\\&\quad+\E_{h_{t+1}\in[-M,M]}\Bigl\vert \E_{n\in[N]}\Big\langle S^{a_{1}n}\prod_{\e'\in\{0,1\}^{t+1}}\Big(\prod_{i=1}^{t+1}S^{b_{i}h_{i}\e_{i}}\Big)\mathcal{C}^{\vert \e'\vert}F_{1}\cdot\prod_{m=2}^{\ell-t-1}S^{a_{m}n}\Big(S^{h_{t+1}}\overline{F}_{m}\cdot F_{m}\Big),
\\&\qquad\qquad\quad\quad\quad\qquad\qquad\quad\quad\quad  S^{a_{\ell-t}n}\Big(S^{h_{t+1}}F_{\ell-t}\cdot \overline{F}_{\ell-t}\Big) \Big\rangle\Bigr\vert^{\kappa}\Bigr).
\end{align*}
Composing by $S^{-a_{\ell-t}n},$ we get that the previous quantity is equal to
\begin{align*}
 &\E_{\bold{h}\in [-M,M]^{t}}\sup_{\Vert F_{2}\Vert_{\infty},\dots,\Vert F_{\ell-t}\Vert_{\infty}\leq 1}\Bigl(\frac{1}{M}+\Big(\frac{M}{N}\Big)^{\kappa}
\\&\qquad+\E_{h_{t+1}\in [-M,M]}\Bigl\vert \E_{n\in[N]}\Big\langle \Big(S^{(a_{1}-a_{\ell-t})n}\prod_{\e'\in\{0,1\}^{t+1}}\prod_{i=1}^{t+1}S^{b_{i}h_{i}\e_{i}}\Big)\mathcal{C}^{\vert \e'\vert}F_{1}
\\&\qquad\quad\cdot\prod_{m=2}^{\ell-t-1}S^{(a_{m}-a_{\ell-t})n}\Big(S^{h_{t+1}}\overline{F}_{m}\cdot F_{m}\Big),  S^{h_{t+1}}F_{\ell-t}\cdot \overline{F}_{\ell-t}
\Big\rangle\Bigr\vert^{\kappa}\Bigr).  
\end{align*}
By the Cauchy-Schwarz inequality we can bound this expression by $O_{\kappa}(1)$ times
\begin{align*}
&\E_{\bold{h}\in [-M,M]^{t}}\sup_{\Vert F_{2}\Vert_{\infty},\dots,\Vert F_{\ell-t}\Vert_{\infty}\leq 1}\Bigl(\frac{1}{M}+\Big(\frac{M}{N}\Big)^{\kappa}
\\&\qquad+\E_{ h_{t+1} \in[-M,M]}\Bigl\Vert \E_{n\in[N]}  \Big(S^{(a_{1}-a_{\ell-t})n}\prod_{\e'\in\{0,1\}^{t+1}}\prod_{i=1}^{t+1}S^{b_{i}h_{i}\e_{i}}\Big)\mathcal{C}^{\vert \e'\vert}F_{1}
\\&\qquad\qquad\quad\cdot\prod_{m=2}^{\ell-t-1}S^{(a_{m}-a_{\ell-t})n}\Big(S^{h_{t+1}}\overline{F}_{m}\cdot F_{m}\Big)
\Bigr\Vert_{2}^{\kappa}\Bigr)  
\\&\leq\E_{\bold{h}'\in [-M,M]^{t+1}}\sup_{\Vert F_{2}\Vert_{\infty},\dots,\Vert F_{\ell-t-1}\Vert_{\infty}\leq 1}
 \Bigl\Vert \E_{n\in[N]}  \Big(S^{(a_{1}-a_{\ell-t})n}\prod_{\e'\in\{0,1\}^{t+1}}\prod_{i=1}^{t+1}S^{b_{i}h_{i}\e_{i}}\Big)\mathcal{C}^{\vert \e'\vert}F_{1}
\\&\qquad\qquad\quad\quad\quad\quad \cdot\prod_{m=2}^{\ell-t-1}S^{(a_{m}-a_{\ell-t})n}F_{m}
\Bigr\Vert_{2}^{\kappa}+\frac{1}{M}+\Big(\frac{M}{N}\Big)^{\kappa}.
\end{align*}
This proves \eqref{eq_s_tilde}.

Using the inequality \eqref{eq_s_tilde} repeatedly, starting from $\tilde{S}(0,2^{\ell},(k_i)_{i=1}^{\ell},\emptyset)$, and keeping track of the coefficients of $a_{i}$, we deduce that the $2^{\ell}$-th power of \eqref{E:base_case_k=1_2 g e=0 6} is bounded by $O_{\ell}(1)$ times 
\[\frac{1}{M}+O_{\ell}\Big(\frac{M}{N}\Big)+\E_{\bold{h}\in [-M,M]^{\ell-1}}\norm{ \E_{n\in[N]} \Big(S^{b_{\ell}n}\prod_{\e\in\{0,1\}^{\ell-1}}\prod_{i=1}^{\ell-1}S^{b_{i}h_{i}\e_{i}}\Big)\mathcal{C}^{\vert \e\vert}F_{1}}_{2}^{2},\]
where $(b_{1},\dots,b_{\ell})=(k_{1},k_{1}-k_{2},\dots,k_{1}-k_{\ell})$. 

Using  
Lemma \ref{hkvdc}, we get
\begin{equation}\label{aavv2}
\begin{split}
&\E_{\bold{h}\in [-M,M]^{\ell-1}}\Bigl\Vert \E_{n\in[N]} \Big(S^{b_{\ell}n}\prod_{\e\in\{0,1\}^{\ell-1}}\prod_{i=1}^{\ell-1}S^{b_{i}h_{i}\e_{i}}\Big)\mathcal{C}^{\vert \e\vert}F_{1}\Bigr\Vert_{2}^{2}
\\&\leq \frac{6M}{N}+\E_{\bold{h}\in [-M,M]^{\ell-1}}\E_{x,y\in [M]}\Bigl(\E_{n\in[N]} \Big\langle \Big(S^{b_{\ell}(n+x)}\prod_{\e\in\{0,1\}^{\ell-1}}\prod_{i=1}^{\ell-1}S^{b_{i}h_{i}\e_{i}}\Big)\mathcal{C}^{\vert \e\vert}F_{1},
\\&\qquad\qquad\qquad\qquad\qquad\qquad\qquad\qquad\qquad\qquad\qquad \Big(S^{b_{\ell}(n+y)}\prod_{\e\in\{0,1\}^{\ell-1}}\prod_{i=1}^{\ell-1}S^{b_{i}h_{i}\e_{i}}\Big)\mathcal{C}^{\vert \e\vert}F_{1}\Big\rangle\Bigr).
\end{split}
\end{equation}
Composing by $S^{-b_{\ell}(n+y)},$ the last line of \eqref{aavv2} is equal to 
\begin{equation}\label{aavv3}
\begin{split}
\\&\frac{6M}{N}+\E_{\bold{h}\in [-M,M]^{\ell-1}}\E_{x,y\in [M]}\Bigl(\E_{n\in[N]} \Big\langle \Big(S^{b_{\ell}(x-y)}\prod_{\e\in\{0,1\}^{\ell-1}}\prod_{i=1}^{\ell-1}S^{b_{i}h_{i}\e_{i}}\Big)\mathcal{C}^{\vert \e\vert}F_{1},
\\&\qquad\qquad\qquad\qquad\qquad\qquad\qquad\qquad\qquad\qquad\qquad\qquad\qquad \prod_{\e\in\{0,1\}^{\ell-1}}\Big(\prod_{i=1}^{\ell-1}S^{b_{i}h_{i}\e_{i}}\Big)\mathcal{C}^{\vert \e\vert}F_{1}\rangle\Bigr)
\\&=\frac{6M}{N}+\E_{\bold{h}'\in [-M,M]^{\ell}}\frac{(2M+1)(M+1-\vert h_{\ell}\vert)}{(M+1)^{2}}\cdot \int_{X\times X} \prod_{\e'\in\{0,1\}^{\ell}}\Big(\prod_{i=1}^{\ell}S^{b_{i}h_{i}\e_{i}}\Big)\mathcal{C}^{\vert \e'\vert}F_{1}\,d(\mu\times\mu)
\\&= \frac{6M}{N}+ O\Big(\frac{1}{M}\Big) +2\E_{\bold{h}'\in [-M,M]^{\ell}}\frac{M+1-\vert h_{\ell}\vert}{M+1}\cdot \int_{X\times X} \prod_{\e'\in\{0,1\}^{\ell}}\Big(\prod_{i=1}^{\ell}S^{b_{i}h_{i}\e_{i}}\Big)\mathcal{C}^{\vert \e'\vert}F_{1}\,d(\mu\times\mu).
\end{split}
\end{equation}
Consider the iterated average 
\begin{equation}\label{aavv}
\begin{split}
\mathbb{E}_{h_{1},\dots,h_{\ell}\in\mathbb{Z}}\int_{X\times X} \prod_{\e'\in\{0,1\}^{\ell}}\Big(\prod_{i=1}^{\ell}S^{b_{i}h_{i}\e_{i}}\Big)\mathcal{C}^{\vert \e'\vert}F_{1}\,d(\mu\times\mu).
\end{split}
\end{equation}
 Inductively, using \cite[Lemma 2.4 (iii)]{DKS}, we have that
  (\ref{aavv}) equals to $\nnorm{F_{1}}_{S^{b_{1}},\dots,S^{b_{\ell}}}^{2^{\ell}}.$  
Using Lemma~\ref{nhn} (repeatedly for (\ref{aavv})), the mean ergodic theorem, and the definition of Host-Kra seminorms, we have that the last line of \eqref{aavv3}  can be bounded by
\[2\nnorm{F_{1}}_{S^{b_{1}},\dots,S^{b_{\ell}}}^{2^{\ell}}+\frac{6M}{N}+ O\Big(\frac{1}{M}\Big).\]
By \cite[Lemma 3.4]{dfks}, \[\nnorm{F_{1}}_{S^{k_{1}},S^{k_{1}-k_{2}},\dots,S^{k_{1}-k_{\ell}}}=\nnorm{f_1\otimes \overline{f_1}}_{S^{k_{1}},S^{k_{1}-k_{2}},\dots,S^{k_{1}-k_{\ell}}}\leq \nnorm{f_{1}}^2_{T^{k_{1}},T^{k_{1}},T^{k_{1}-k_{2}},\dots,T^{k_{1}-k_{\ell}}}.\]
By first letting $N\to\infty$ and then $M\to\infty$,  we deduce  that \eqref{E:base_case_k=1_3} is bounded by a constant, depending only on $\ell,$ times  $\nnorm{f_1}_{T^{k_{1}},T^{k_{1}},T^{k_{1}-k_{2}},\dots,T^{k_{1}-k_{\ell}}},$ as was to be shown.  
\end{proof}

\begin{theorem}\label{polytolinear} 
Let $\ell,\kappa\in\mathbb{N}$, $A=(0,\ell,\p)$ be a 1-standard PET-tuple with $\p=(p_{N,1},\dots,p_{N,\ell})_{N}$ of the form
\[p_{N,j}(n)=e_{j} n^{K}+p'_{N,j}(n), \;\;1\leq j\leq \ell,\]
for some $K\in\N$ and variable polynomials $p'_{N,j}\colon\Z\to\R^{\ell}$ of degree less than $K$, when $N$ is sufficiently large. Let $(X,\mathcal{B},\mu,(T_{n})_{n\in \Z^d})$ be a $\Z^d$-system and $f\in L^{\infty}(\mu)$.
Then, there exists $D=O_{\ell,K}(1)$ such that   \[\text{if}\;\;\nnorm{f}_{\Big\{T_{e_{1}}^{\times D},(T_{e_{1}-e_{j}})^{\times D}\Big\}_{1\leq j\leq \ell,j\neq 1}}=0,\;\;\text{then we have that}\;\; S(A,f,\kappa)=0.\] 
Moreover, in the special case where $\ell=1$, 
\[\text{if}\;\;\E(f\otimes \overline{f} \mid I((T_1 \times T_1)^{a}))=0,\;\;\text{for all } a\in \Z\setminus \{0\} {,~   then}\;\;S(A,f,\kappa)=0.\footnote{ In particular, if $\nnorm{f}_{T_1,T_1}=0$, then we have $S(A,f,\kappa)=0$ by \cite[Lemma~3.4]{dfks} and \cite[Lemma 2.4 (iv)]{DKS}.}\]
\end{theorem}
		
\begin{proof}
If $\ell=1$, 
then denote $A':=\partial_{1}\dots\partial_{1}A$ with $\partial_{1}$ repeated $K-1$ times.
It is not hard to compute that  $A'=(K-1,1,(q_{N})_{N})$, where
\[q_{N}(n;h_{1},\dots,h_{K-1})=e_{1}K!h_{1}\dots h_{K-1}n+r_{N}(h_{1},\dots,h_{K-1})\]
for some $r_{N}(h_{1},\dots,h_{K-1})\in\mathbb{R}^d$ 
   when $N$ is sufficiently large.
			By Lemma \ref{1234}, $S(A,f,2^{K-1})\ll_{K} S(A',f,1)$.  By
			  Proposition~\ref{polytolinear2}, and the assumption that $\E(f\otimes \overline{f} \mid I((T_1 \times T_1)^{a}))=0$ for all $a\neq 0$, we get that  \[\limsup_{N \to \infty} \sup_{\vert c_{n}\vert\leq 1}\norm{\mathbb{E}_{0\leq n\leq N}c_{n}T^{[K!h_{1}\dots h_{K-1} n+r_{N}(h_{1},\dots,h_{K-1})]}f}_{2}=0,\]
     provided that $h_1\dots h_{K-1}\neq 0$. So $S(A',f,1)=0$ since the set of $(h_{1},\dots,h_{K-1})$ such that $h_{1}\dots h_{K-1}=0$ is of zero density. So $S(A,f,2^{K-1})=0$, which implies that $S(A,f,\kappa)=0$.
			  			
We now consider the case $\ell>1$. By Lemma \ref{reduction}, there exist $r\in\mathbb{N}$ depending only on $K$ and $\ell$ and $i_{1},\dots,i_{r}\in\mathbb{N}$ such that writing, $A'\coloneqq \partial_{i_r}\dots\partial_{i_1}A,$ 
		    we have that $\deg(A')=1$ and $A'$ is 1-standard and non-degenerate, and that each step $\partial_{i_{t-1}}\dots\partial_{i_1}A\to \partial_{i_{t}}\dots\partial_{i_1}A$ is 1-inherited. 
       By Lemma \ref{1234}, in order to show that  $S(A,f,2^{r})=0$, it suffices to show that $S(A',f,1)=0$. 
      
      Assume that $A'=(r,\ell',(q_{N,m})_{N})$.
      Let 
$\p\coloneqq (p_{N}(n)e_{1},\dots,p_{N}(n)e_{d})_{N}$
denote the initial tuple of variable polynomials in the PET-tuple $A$ and $\q=(q_{N,1},\dots,q_{N,\ell'})_{N}$ denote the tuple of variable  polynomials in the PET-tuple $A'$. Since $\p$ satisfies (P1)--(P4) with respect to $\p$, by Proposition~\ref{PET2}, $\q$ also satisfies (P1)--(P4) with respect to $\p$. 

 Since $\deg(A')=1$,
      we may assume that 
			\begin{equation}\label{qc}
			q_{N,m}(n,h_1,\ldots,h_r)=c_{N,m}(h_1,\ldots,h_r) n+r_{N,m}(h_1,\ldots,h_r)
			\end{equation}
			for some polynomials $c_{N,m}\colon \mathbb{Z}^{r}\to\mathbb{R}^d$ with degree, in terms of the variables $h_1,\ldots,h_r,$ less than $K$\footnote{Here we used the obvious fact that if $A$ is of degree at most $K$ in terms of all the variables $n,h_{i},$ then so is $\partial_{t}A$.} and some  $r_{N,m}(h_1,\ldots,h_r)\in\mathbb{R}^d$ when $N$ is sufficiently large.

\smallskip

\noindent      \textbf{Claim.} For all $1\leq m\leq \ell'$, every $c_{N,m}$ is equal to the same polynomial $c_{m}\colon \mathbb{Z}^{r}\to\mathbb{Z}^d$ when $N$ is sufficiently large. 

   Write 
      $$q_{N,m}(n;h_{1},\dots,h_{s})=\sum_{b,a_{1},\dots,a_{s}\in\mathbb{N}_{0},b+a_{1}+\dots+a_{s}\leq K}u_{N,m}(b,a_{1},\dots,a_{s})n^{b}h_{1}^{a_{1}}\dots h_{s}^{a_{s}}$$
and
 \[p_{N,m}(n)=\sum_{v\in\mathbb{N}_{0},v\leq K}b_{N,m,v}n^{v}\] 
   for some $u_{N,m}(b,a_{1},\dots,a_{s})\in\mathbb{R}^{d}$,  $b_{N,m,v}\in\mathbb{R}^{d}$ 
 for all $1\leq m\leq \ell'$.
    It suffices to show that for all $a_{1},\dots,a_{s}\in\mathbb{N}_{0}$ and  $1\leq m\leq \ell'$, $u_{N,m}(1,a_{1},\dots,a_{s})$ equals to a same vector in $\Z^{d}$ when $N$ is large enough. 
 
  We may assume that each $\u(\q,1;a_{1},\dots,a_{s})$ is associated with a type and a symbol so that (P1)--(P4) hold.
Fix any $a_{1},\dots,a_{s}\in\mathbb{N}_{0}$. By (P1) and (P2), we may assume that $\u(\q,1;a_{1},\dots,a_{s})$ is associated with the type $(r,i,v)$ and symbol $(1,w_{2},\dots,w_{\ell'})$, where $r=\binom{1+a_{1}+\dots+a_{s}}{1,a_{1},\dots,a_{s}}$ and $v=1+a_{1}+\dots+a_{s}$.  If all of $u_{N,m}(\q,1;a_{1},\dots,a_{s}),1\leq m\leq \ell'$ are $\bold{0}$ when $N$ is sufficiently large, then we are done. If not, then there exists $1\leq m\leq \ell'$ such that $u_{N,m}(\q,1;a_{1},\dots,a_{s})=r(b_{N,w_{m},v}-b_{N,i,v})$ is not constant $\bold{0}$ when $N$ is sufficiently large. Then $w_{m}\neq i$ and $v\leq K$. 

 If $v=K$, then $u_{N,m}(\q,1;a_{1},\dots,a_{s})=r(b_{N,w_{m},K}-b_{N,i,K})=r(e_{w_{m}}-e_{i})\neq \Z^{d}\backslash\{\bold{0}\}$ when $N$ is sufficiently large and we are done.

If $v<K$, then setting $b'=K-v+1$, we have $b'\geq 2$. By (P3), $\u(\q,b';a_{1},\dots,a_{s})$ is of type $(r',i,v')$, for some $v'$, where $r'=\binom{v'}{b',a_{1},\dots,a_{s}}\neq 0$. By (P2), $v'=b'+a_1,\ldots,a_s=K$. Since $w_{m}\neq i$, we conclude that $u_{N,m}(\q,b';a_{1},\dots,a_{s})=r'(b_{N,w_{m},K}-b_{N,i,K})=r'(e_{w_{m}}-e_{i})\neq \bold{0}$ when $N$ is sufficiently large, a contradiction to the fact that $\deg(\q)=1$. This completes the proof of the claim.  

Denote $\bold{h}:=(h_1,\ldots,h_r)$. By the claim,
we have	
\begin{equation}\nonumber
\begin{split}
S(A',f,1)=\Es_{\bold{h}\in\Z^{r}}\limsup_{N \to \infty} \sup_{\vert c_{n}\vert\leq 1}\sup_{\norm{g_2}_\infty,\ldots,\norm{g_r}_\infty\leq 1}\norm{\mathbb{E}_{0\leq n\leq N}c_{n}\prod_{m=1}^{\ell'} T^{c_m(\bold{h}) n+[r_{N,m}(\bold{h})]}g_m}_{2},
\end{split}
\end{equation}
			where $g_1=f$. 
				By Proposition \ref{polytolinear2}, 
				$$\limsup_{N \to \infty} \sup_{\vert c_{n}\vert\leq 1}\sup_{\norm{g_2}_\infty,\ldots,\norm{g_r}_\infty\leq 1}\norm{\mathbb{E}_{0\leq n\leq N}c_{n}\prod_{m=1}^{\ell'} T^{c_m(\bold{h}) n+[r_{N,m}(\bold{h})]}g_m}_{2}$$
				is bounded by  $C\cdot \nnorm{f}_{T^{c_{1}(\bold{h})},T^{c_{1}(\bold{h})},T^{c_{1}(\bold{h})-c_{2}(\bold{h})},\dots,T^{c_{1}(\bold{h})-c_{\ell}(\bold{h})}},$ where $C$ depends only on $\ell'$, which can be bounded in terms of $\ell$ and $K.$\footnote{ We remark that this is where we crucially used the fact that Proposition \ref{polytolinear2} depends only on the number of linear iterates.}
				So $S(A',f,1)$ is bounded by
    \begin{equation}\label{badC}
C\cdot\Es_{\bold{h}\in\mathbb{Z}^{r}}\nnorm{f_{1}}_{T^{c_{1}(\bold{h})},T^{c_{1}(\bold{h})},T^{c_{1}(\bold{h})-c_{2}(\bold{h})},\dots,T^{c_{1}(\bold{h})-c_{\ell}(\bold{h})}}.
    \end{equation}
				Assume that
				$$c_{m}(\bold{h})=\sum_{a_{1},\dots,a_{r}\in\mathbb{N}_{0}, a_{1}+\dots+a_{r}\leq K}h_{1}^{a_{1}}\dots h_{r}^{a_{r}}u_{m}(a_{1},\dots,a_{r})$$
				for some 
				$u_{m}(a_{1},\dots,a_{r})\in\mathbb{Q}^{d}$.
	Let 
				$$H_{m}\coloneqq G(u_{1}(a_{1},\dots,a_{r})-u_{m}(a_{1},\dots,a_{r})\colon a_{1},\dots,a_{r}\in\mathbb{N}_{0}).$$
				for $0\leq m\leq \ell, m\neq 1$.
				By \cite[Proposition~5.2]{dfks}, there exists $D\in\mathbb{N}$ depending only on $\ell$ and $K$ such that
				\[\Es_{\bold{h}\in\mathbb{Z}^{r}}\nnorm{f_{1}}_{T^{c_{1}(\bold{h})},T^{c_{1}(\bold{h})},T^{c_{1}(\bold{h})-c_{2}(\bold{h})},\dots,T^{c_{1}(\bold{h})-c_{\ell}(\bold{h})}}=0\]
				if $\nnorm{f}_{H_{m}^{\times D},0\leq m\leq \ell, m\neq 1}=0$. \footnote{  We remark at this point that it is \cite[Proposition~5.2]{dfks} that crucially uses concatenation results from \cite{TZ}.}

Since $\q$  satisfies (P1)--(P4) with respect to $\p$,
by Proposition \ref{12-3}, for all $0\leq m\leq r, m\neq 1$, $H_{N,1,m}(\q)$ contains one of $G'_{N,1,j}(\p), 0\leq j\leq d, j\neq 1$ for all $N$ sufficiently large.
In our case,
$$H_{N,1,m}(\q)=\text{span}_{\mathbb{Q}}(u_{1}(a_{1},\dots,a_{r})-u_{m}(a_{1},\dots,a_{r})\colon a_{1},\dots,a_{r}\in\mathbb{N}_{0})$$
 and $$G'_{N,1,j}(\p)=\text{span}_{\mathbb{Q}}\{e_{1}-e_{j}\}.$$
 So, each $H_{m}:=H_{N,1,m}(\q)\cap\Z^{d}$ contains one of $\text{span}_{\mathbb{Z}}\{e_{1}-e_{j}\}, 0\leq j\leq d, j\neq 1$.
Hence, 
if $\nnorm{f}_{T^{\times D'}_{e_{1}},T^{\times D'}_{e_{2}-e_{1}},\dots,T^{\times D'}_{e_{d}-e_{1}}}=0$, where $D'=Dd$, then as a consequence of \cite[Lemma 2.4 (v)]{DKS} we have
$\nnorm{f}_{H_{m}^{\times D},0\leq m\leq \ell, m\neq 1}=0$, and thus the average is 0.
\end{proof}

\begin{remark}\label{R:cdwp}
\emph{The reason why we cannot obtain the more general assumption $\log \prec h$ instead of $\log  \prec s_h$ in Theorem~\ref{t1_2} (in which case we would also cover the case where $h$ is a polynomial), is that 
 we cannot extend Theorem~\ref{polytolinear} when the leading coefficients of the $p_{N,j}$'s 
 are equal to some non-integer $\alpha$. In this case, we are able to obtain an upper bound for $S(A',f,1)$ similar to (\ref{badC}), but with the constant $C$ depending on $\bold{h}$. This prevent us from using \cite[Proposition~5.2]{dfks} or any concatenation result from \cite{TZ} to get a satisfactory estimate.
}
\end{remark}

\section{Proof of main result}\label{Sec_MR}

We prove Theorem~\ref{t1_2} in this section. Combining the estimates in the previous section, we first provide a Host-Kra seminorm upper bound for multiple ergodic averages with non-polynomial iterates. 

\begin{theorem}\label{T:upper bound}
Let $(X,\mathcal{B},\mu,T_{1},\dots,T_{d})$ be a system with commuting and invertible transformations,
$a$ be a function, $L\in \mathcal{H}$ a positive function with $1\prec L(x)\prec x,$ and
 $(p_{N})_{N}$ a sequence of functions 
   such that for all $N\in \mathbb{N}$ and $0\leq r\leq L(N),$  we have
			\[a(N+r)=p_N(r)+e_{N,r},\;\; \text{with}\;\; e_{N,r}\ll 1.\]
Assume additionally that $p_{N}$ are polynomials such that, when $N$ is sufficiently large, $\deg(p_{N})=K,$ for some $K\in\mathbb{N},$ and the leading coefficient of $p_N$ equals to $a_N:=a^{(K)}(N)/K!,$ where we have 
\[\lim_{N\to\infty} L(N)|a_{N}|^{\frac{1}{K}}=\infty,\;\;
    \lim_{N\to\infty}a_{N}= 0,\;\;\text{ and}\;\;
    L(N) \ll |a_{N}|^{-\frac{K+1}{K^{2}}}.\]
 There exists $D\in\mathbb{N}$ depending only on $K$ and $d$ such that if $\nnorm{f_{1}}_{(T_{1},T_{1}T^{-1}_{2},\dots,T_{1}T^{-1}_{d})^{\times D}}=0$, then
			\begin{equation}\label{ert}
			\limsup_{N\to\infty}\norm{\mathbb{E}_{1\leq n\leq N}T_{1}^{[a(n)]}f_{1}\cdot\ldots\cdot T_{d}^{[a(n)]}f_{d}}_2=0.
			\end{equation}
	Moreover, when $d=1$, (\ref{ert}) holds if $\| \E(f_1\otimes \overline{f_{1}} | I((T_1\times T_1)^{a}))\|_2=0$ for all $a\in\mathbb{Z}\backslash\{0\}$. 
	
(A similar result holds if $f_{1}$ is replaced by any of the $f_{2},\dots,f_{d}$.)
\end{theorem}

We briefly explain the idea of the proof of Theorem~\ref{T:upper bound} using the Examples~\ref{E:1} and ~\ref{E:2}.
We have already seen (in Section~\ref{Sec_2}) that for these examples, the Hardy field iterates can be approximated by variable polynomials that can be transformed in such a way that their leading coefficients are equal to 1. Then, we may use Theorem~\ref{polytolinear} to get the desired seminorm control.

\begin{proof}[Proof of Theorem~\ref{T:upper bound}]
Since $a(N+r)=p_N(r)+e_{N,r},$ $N\in \mathbb{N},$ $0\leq r\leq L(N),$ it suffices by Proposition~\ref{P:HostSemi} to show
 \begin{equation*}
		\limsup_{N \to \infty}\sup_{\vert c_{n}\vert\leq 1}\sup_{\Vert f_{2}\Vert_{\infty},\dots,\Vert f_{d}\Vert_{\infty}\leq 1}\norm{\mathbb{E}_{0\leq n\leq L(N)}c_{n}\prod_{i=1}^{d}T_{i}^{[p_N(n)]}f_{i}}_{2}=0.
			\end{equation*}
Since $p_N(n)=a_N n^K+p'_N(n),$ for some $(p'_N)_N$ of degree less than $K$ and $(a_N)_N, L$ satisfy the assumptions of Proposition~\ref{P:cov}, it suffices to show 
\begin{equation}\label{42}
		\limsup_{N \to \infty}\sup_{\vert c_{n}\vert\leq 1}\sup_{\Vert f_{2}\Vert_{\infty},\dots,\Vert f_{d}\Vert_{\infty}\leq 1}\norm{\mathbb{E}_{0\leq n\leq \tilde{L}(N)}c_{n}\prod_{i=1}^{d}T_{i}^{[n^K+\tilde{p}_N(n)]}f_{i}}_{2}=0
			\end{equation}
for the appropriate $(\tilde{p}_{N})_N$ of degree less than $K$ and the positive function $\tilde{L}$ with $1\prec \tilde{L}(x)\prec x$ given by Proposition~\ref{P:cov}. 
  
 If (\ref{42}) fails, then there exist $\epsilon>0$ and a subsequence $(N_{j})_{j}$ of integers such that 
\begin{equation}\label{422}
 \sup_{\vert c_{n}\vert\leq 1}\sup_{\Vert f_{2}\Vert_{\infty},\dots,\Vert f_{d}\Vert_{\infty}\leq 1}\norm{\mathbb{E}_{0\leq n\leq \tilde{L}(N_{j})}c_{n}\prod_{i=1}^{d}T_{i}^{[n^{K}+\tilde{p}_{N_{j}}(n)]}f_{i}}_{2}>\epsilon
\end{equation}
for all $j\in\mathbb{N}$. Passing to another subsequence if necessary, we may assume without loss of generality that $M_{j}:=[\tilde{L}(N_{j})]$ is strictly increasing in $j$.

Let $A=(0,d,{\bf{q}})$ be the 1-standard PET-tuple given by $q_{N,i}(n)=(n^{K}+q_{N}(n))e_{i}, 1\leq i\leq d,$ $N\in\N$, where $q_{M_{j}}\coloneqq \tilde{p}_{N_{j}}$ and $q_{N}\coloneqq 0$ otherwise.
By Theorem~\ref{polytolinear},
	for $d>1$, $S(A,f_{1},1)=0$ if $\Vert f_{1}\Vert_{(T_{1},T_{1}T^{-1}_{2},\dots,T_{1}T^{-1}_{d})^{\times D}}=0$. In the $d=1$ case, the same theorem implies $S(A,f_1,1)=0$ if $\norm{\E(f_1\otimes \overline{f_1} | I((T_1\times T_1)^{a}))}_2=0$ for all $a\in\mathbb{Z}\backslash\{0\}$. 
	
By the construction of $A$,
in both cases, we have that the left hand side of (\ref{422}) converges to 0, a contradiction.
This finishes the proof.
\end{proof}

\begin{remark}\label{R:5.1}
\emph{The Hardy functions of interest satisfy the conclusion of Theorem~\ref{T:upper bound} (for some appropriate positive function $L\in \mathcal{H}$ and $K\in \mathbb{N}$).}

\emph{Indeed,  let $h(x)=s_h(x)+p_h(x)+e_h(x)$ with $\log \prec s_h.$ Since we can drop the bounded error terms (see also the expression of $h$ via variable polynomials below), it suffices to deal with the case
  \begin{equation*}
  h(x)=s_h(x)+p_h(x).
  \end{equation*} 
  Let $d_{p_h}$ be the degree of $p_{h}$ and $d_{s_h}$ be the degree of $s_{h}$.
If $d_{p_h}<d_{s_h}+1,$ we set $K:=d_{s_h}+1,$ while if $d_{p_h}\geq d_{s_h}+1,$ we set $K:=d_{p_h}+1.$ By \cite[Proposition~A.2]{Ts}\footnote{ We can use \cite[Proposition~A.2]{Ts} since $\log \prec s_h.$} we have that:
\[1\prec\vert s_h^{(K)}(x)\vert^{-\frac{1}{K}}\prec \vert s_h^{(K+1)}(x)\vert^{-\frac{1}{K+1}}\prec x,\]
and since $s_h^{(K+1)}$ is a Hardy field function, it is (eventually) monotone.}

\

\emph{By the previous relation, we may choose $L\in \mathcal{H}$ such that 
\[1\prec\vert s_h^{(K)}(x)\vert^{-\frac{1}{K}}\prec L(x)\prec \min\big\{\vert s_h^{(K+1)}(x)\vert^{-\frac{1}{K+1}},\vert s_h^{(K)}(x)\vert^{-\frac{K+1}{K^{2}}}\big\}\]
(for example take the geometric mean of the   functions appearing above).\footnote{Every two Hardy field functions are comparable, hence the minimum of the right hand side is (eventually) one of the functions.}} 

\emph{Then, by the Taylor expansion, for all $N,r\in\mathbb{N}_{0}$, there exists $\xi_{N,r}\in[N,N+r],$ such that
\[s_h(N+r)=s_h(N)+\dots+\frac{s_h^{(K)}(N)}{K!}r^{K}+\frac{s_h^{(K+1)}(\xi_{N,r})}{(K+1)!}r^{K+1}.\] If $0\leq r\leq L(N)$, then, for $N$ sufficiently large, using the monotonicity of $s_h^{(K+1)},$ we have that
\[\Bigl\vert\frac{s_h^{(K+1)}(\xi_{N,r})r^{K+1}}{(K+1)!}\Bigr\vert\leq \Bigl\vert\frac{s_h^{(K+1)}(N)r^{K+1}}{(K+1)!}\Bigr\vert\leq \Bigl\vert\frac{s_h^{(K+1)}(N)L(N)^{K+1}}{(K+1)!}\Bigr\vert\ll 1.\] 
		 Denoting
		\[p_{N}(r):=p_h(N+r)+s_h(N)+\dots+\frac{s_h^{(K)}(N)}{K!}r^{K},\] we have \[[h(N+r)]=[p_N(r)]+e_{N,r},\] with $e_{N,r}\ll 1$ (here we can also absorb the initial $e_h$ term). Since the assumptions of both Proposition~\ref{P:HostSemi} and ~\ref{P:cov} are satisfied for $h, L, p_N$ and $(s_h^{(K)}(N)/K!)_N, L, K$ (noticing that for all $k\geq K$ we have $h^{(k)}=s_h^{(k)}$) respectively, the function $h$ satisfies the conclusion of Theorem~\ref{T:upper bound}.}
  \end{remark}

	We will now show
	that conditions (i) and (ii) of Theorem~\ref{t1_2} are implied by the joint ergodicity of  $(T_1^{[h(n)]})_n,\dots,(T_d^{[h(n)]})_n$.
	
\begin{proof}[Proof of the necessity of conditions \emph{(i)} and \emph{(ii)} in Theorem~\ref{t1_2}]
To show (i), for any $1 \leq i,j \leq d,$ $ i\neq j,$ setting $f_k=1$ for $k \neq \{i,j\}$, and since strong convergence implies weak convergence, we see that
	\[ \lim_{N \to \infty} \mathbb{E}_{n \in [N]} \int_X (T_iT_j^{-1})^{[h(n)]}f_i \cdot f_j \ d\mu=\int_X f_i \ d \mu \int_X f_j \ d\mu\]
	for all $f_i, f_j \in L^{\infty}(\mu)$. Thus, $((T_iT_j^{-1})^{[h(n)]})_{n}$ is an ergodic sequence, as desired.
	
	To prove (ii), it suffices to show that for any $f \in L^{\infty}(\mu^{\otimes d})$
	\begin{equation}\label{necessity1} \lim_{N \to \infty} \mathbb{E}_{n \in [N]} (T_1 \times \dots \times T_d)^{[h(n)]} f=\int_X f \ d\mu^{\otimes d},
	\end{equation}
	where convergence takes place in $L^2(\mu^{\otimes d})$.
	By a standard  linearity and density argument, it suffices to prove (\ref{necessity1}) for the case $f=f_{1}\otimes\dots\otimes f_{d}$ for some $f_{1},\dots,f_{d}\in L^{\infty}(\mu)$. 

  We claim that both sides of \eqref{necessity1}  are equal to 0 if $\nnorm{f_{i}}_{T_{i},T_{i}}=0$ for some $1\leq i\leq d$. 
  Assume that $\nnorm{f_i}_{T_i,T_i}=0.$ By \cite[Lemma 2.4 (iv)]{DKS},   $\nnorm{f_i}_{T_i^a,T_i^a}=0$ for all $a\neq 0$. By the proof of \cite[Lemma~5.2]{DKS}, this implies that $\E(f\otimes \overline{f} | I((S\times S)^{a}))\|_2=0$. Since $h$ satisfies the conclusion of Theorem~\ref{T:upper bound}, we have that the left hand side of \eqref{necessity1} is 0.
  
On the other hand, $\nnorm{f_{i}}_{T_{i},T_{i}}=0,$ implies that \[\left(\int_{X} |\mathbb{E}(f_i|I(T_i))|^2 d\mu\right)^{1/2}=\nnorm{f_{i}}_{T_{i}}\leq \nnorm{f_{i}}_{T_{i},T_{i}}=0,\] which in turn implies that $\int_{X}f_{i}\,d\mu=\int_{X}\mathbb{E}(f_{i}\vert I(T_i))\,d\mu=0$; thus the right hand side of \eqref{necessity1} is 0. 

Therefore, it suffices to  prove \eqref{necessity1} under the assumption that 
each $f_{i}$ is measurable with respect to $Z_{T_{i},T_{i}}$, the sub $\sigma$-algebra of $\mathcal{B}$ such that $\nnorm{f}_{T_i,T_i}=0\Leftrightarrow \mathbb{E}(f\vert Z_{T_{i},T_{i}})=0$ for all $f\in L^{\infty}(\mu)$.

Since $(T_1^{[h(n)]})_n,\dots,(T_d^{[h(n)]})_n$  are jointly ergodic, by projecting to each coordinate, we have that
$T_{i}$ is ergodic for $\mu$ for all $1\leq i\leq d$.
By \cite[Lemma 2.7]{DKS},   we may approximate each $f_{i}$ by finite linear combinations of
eigenfunctions of $T_{i}$. So,  we may assume that for each $f_i$ we have $T_if_i=\lambda_i f_i,$ for some $\lambda_i \in \mathbb{S}^1$.
If one of $f_{1},\dots,f_{d}$ is 0 $\mu$-a.e., then \eqref{necessity1} holds trivially. Suppose now that none of $f_{1},\dots,f_{d}$ is 0 a.e..
Since  
$T_i$ is ergodic for each $i$, it follows that we may assume that $|f_i|=1$ $\mu$-a.e., for each $i.$ If all the $f_i$'s are constant, \eqref{necessity1} holds trivially. If not, say $f_{i_0}$ is not a constant, then $\lambda_{i_{0}}\neq 1$ by the ergodicity of $T_{i_{0}}$. So $\int_X f_{i_0}\;d\mu=0,$ and thus the right hand side of \eqref{necessity1} is $0.$ Consequently, we have reduced matters to showing that
	\[ \lim_{N \to \infty} \mathbb{E}_{n \in [N]} (\lambda_1\cdots \lambda_d)^{[h(n)]}=0.\] 
	This follows directly by the joint ergodicity assumption applied to the eigenfunctions $f_1,\dots,f_d$ described above, completing the proof.
	\end{proof}

In order to show that conditions (i) and (ii) in \cref{t1_2} are sufficient for joint ergodicity, we use a criterion, first introduced by Frantzikinakis \cite{FA}  and then generalized by Best and Ferr\'e Moragues \cite{AA}. 
	
\begin{definition*}[\cite{AA}]
We say that a collection of mappings $a_1,\ldots,a_k\colon \Z^d \to \Z^d$ is:  
\begin{enumerate}
\item[(a)] \emph{good for seminorm estimates for the system $(X,\mathcal{B},\mu, (T_n)_{n\in \Z^d}),$ 
			}
			if there exists $M\in \N$ such that if $f_1,\ldots,f_k\in L^\infty(\mu)$ and $\nnorm{f_\ell}_{(\Z^d)^{\times M}}=0$ for some $\ell\in\{1,\ldots,k\},$ then 
			\begin{equation*}
			\lim_{N\to\infty}\frac{1}{N^{d}}\sum_{n\in [N]^{d}}\prod_{i=1}^{k} T_{a_{i}(n)}f_i=0,
			\end{equation*} where the convergence takes place in $L^2(\mu).$ 
			\item[(b)] \emph{good for equidistribution for the system $(X,\mathcal{B},\mu, (T_n)_{n\in \Z^{d}}),$
			}
			if for every $\alpha_1,\ldots,\alpha_k\in $ $\text{Spec}\left((T_{n})_{n\in \Z^{d}}\right),$ not all of them being trivial, we have 
			\begin{equation*}
			\lim_{N\to\infty}\frac{1}{N^{d}}\sum_{n\in [N]^{d}}\exp(\alpha_{1}(a_{1}(n))+\dots+\alpha_{k}(a_{k}(n)))=0,       
			\end{equation*}
			where $\exp(x)\coloneqq e^{2\pi i x}$ for all  $x\in\R,$ and \[\text{Spec}\left((T_{n})_{n\in\Z^{d}}\right)\coloneqq \{\alpha\in \text{Hom}(\Z^{d},\T)\colon T_{n}f=\exp(\alpha(n))f, \text{ $n\in \Z^{d},$ for some non-zero } f\in L^{2}(\mu)\}.\]  
		\end{enumerate}
	\end{definition*}

It was shown in
\cite[Theorem~1.1]{AA} that for an ergodic system $(X,\mathcal{B},\mu,T_{1},\dots,T_{d})$ (meaning that the group action generated by $T_{1},\dots,T_{d}$ is ergodic), a collection of mappings $a_1,\ldots,a_k\colon \Z^d \to \Z^d,$ $(T_{a_1(n)})_n,\ldots, (T_{a_k(n)})_n$ are  jointly ergodic for $\mu$\footnote{ Here we mean that, for all bounded $f_i$'s, we have $\lim_{N\to\infty}\frac{1}{N^d}\sum_{n\in [N]^d}T_{a_1(n)}f_1\cdot\ldots\cdot T_{a_k(n)}f_k=\prod_{i=1}^k\int f_i\;d\mu.$} if, and only if, they are good for seminorm estimates and good for equidistribution for the system. 

\begin{proof}[Proof of the sufficiency of conditions \emph{(i)} and \emph{(ii)} in Theorem~\ref{t1_2}]
Fix any system $(X,\mathcal{B},\mu, T_{1},$ $\dots,T_{d})$ that satisfies conditions (i) and (ii). 
We use $(X,\mathcal{B},\mu, (T_{n})_{n\in\mathbb{Z}^{d}})$ to denote the $\mathbb{Z}^{d}$-system with $T_{e_{i}}:=T_{i}$ for $1\leq i\leq d$.
	Our goal is to use \cite[Theorem~1.1]{AA} to show the desired joint ergodicity. To do so, we will take 
	$h_i(n_{1},\dots,n_{d})\coloneqq e_ih(n_1)$ for all $(n_{1},\dots,n_{d})\in \N^d$, where $e_{i}$ is the $i$-th canonical vector, (i.e., the $h_i$'s depend only on the first coordinate of $n$). 
 First, note that (i) and (ii) imply that $T_i, T_{i}T_{j}^{-1}$  are ergodic for all $1\leq i,j\leq d, i\neq j$, which also implies that our system is ergodic.
By \cite[Theorem~1.1]{AA}, it suffices to show that for the system $(X,\mathcal{B},\mu, (T_n)_{n\in \Z^d})$, the mappings $h_1,\dots,h_{d}$ are good for seminorm estimates and good for equidistribution. 
	
	The fact that $h_1,\dots,h_d$ are good for seminorm estimates can be argued as follows:
	note that
	\[\lim_{N\to\infty}\frac{1}{N^{d}}\sum_{n\in [N]^{d}}\prod_{i=1}^d T_{h_{i}(n)}f_i=\lim_{N\to\infty}\frac{1}{N}\sum_{n=1}^{N}\prod_{i=1}^d T_{i}^{h(n)}f_i.\]
	Since $h$ satisfies the conclusion of Theorem~\ref{T:upper bound},
	 there exists $D\in\N$ depending only on $d$ and the degree of $h(n)$ such that for all $i \in\{1,\dots,d\},$
	\[ \nnorm{f_i}_{(T_i,(T_{i}T^{-1}_{j})_{j\neq i})^{\times D}}=0 \ \text{ implies that } \  \lim_{N\to\infty}\frac{1}{N^{d}}\sum_{n\in [N]^{d}}\prod_{i=1}^d T_{h_{i}(n)}f_i=0.\]
	By \cite[Corollary~2.5]{DKS}, since $T_{i}, T_{i}T_{j}^{-1}$ are ergodic for all $1\leq i, j\leq d, i\neq j$,  we get that 
	$\nnorm{f_i}_{(T_i,(T_iT_{j}^{-1})_{j\neq i})^{\times D}}=0$ if, and only if, $\nnorm{f_i}_{(\Z^d)^{\times d D}}=0$.  From this we have the good for seminorm estimates condition.	 

 	Thus, it only remains to show that the collection of $h_{1},\dots,h_{d}$ is good for equidistribution.
	Suppose, for the sake of contradiction, that $h_{1},\dots,h_{d}$ is not good for equidistribution. Then, there exist $\alpha_{1},\dots,\alpha_{d}\in \text{Spec}\left((T_{n})_{n\in\Z^{d}}\right)$, not all of them trivial, and a subsequence $(N_{j})_{j\in\mathbb{N}}$ of $\mathbb{N}$, such that 
		\begin{equation}\label{E:d=L}
		\lim_{j\to\infty}\frac{1}{N_{j}^{d}}\sum_{n\in [N_{j}]^{d}}\exp(\alpha_{1}(h_{1}(n))+\dots+\alpha_{d}(h_{d}(n))) \text{ exists and equals to } c,
		\end{equation}
		for some $c\neq 0$.
		For $1\leq i\leq d$,
		since $\alpha_{i} \in \text{Spec}\left((T_{n})_{n\in\Z^{d}}\right)$, there exists some nonzero $f_{i}\in L^{2}(\mu)$ such that $T_{n}f_{i}=\exp(\alpha_{i}(n))f_{i}$ for all $n\in\Z^{d}$. Since $(X,\mathcal{B},\mu, T_{1},\dots,T_{d})$ is ergodic, we have that $\vert f_{i}\vert$ is a non-zero constant $\mu$-a.e.. Using \eqref{E:d=L}, we have
			\begin{equation}\nonumber
	\begin{split}
	   & \lim_{j\to\infty}\frac{1}{N_{j}}\sum_{n=1}^{N_{j}}\bigotimes_{i=1}^d T_{i}^{h(n)}f_i=\lim_{j\to\infty}\frac{1}{N_{j}^{d}}\sum_{n\in [N_{j}]^{d}}\bigotimes_{i=1}^d T_{h_{i}(n)}f_i
	   \\&=\lim_{N\to\infty}\frac{1}{N_{j}^{d}}\sum_{n\in [N_{j}]^{d}}\bigotimes_{i=1}^d \exp(\alpha_{i}(h_{i}(n)))f_i=c\bigotimes_{i=1}^d f_{i}\not\equiv 0.
	    \end{split}
	\end{equation}
		On the other hand, since at least one of $\alpha_{1},\dots,\alpha_{d}$ is non-trivial, we have that $\int_{X^{d}}\bigotimes_{i=1}^{d}f_{i}\,d\mu^{\otimes d}$ $=\prod_{i=1}^{d}\int_{X}f_{i}\,d\mu=0$, which contradicts condition (ii). Therefore, $h_{1},\dots,h_{d}$ are good for equidistribution.
\end{proof}

\section{An application of the method to more general iterates}\label{Section:7}

 In this section, we extend  Theorem~\ref{t1_2} to a wider class of functions.   

\begin{definition*}[Tempered functions]
    Let $i\in\mathbb{N}_{0}$. A real-valued function $t$ which is $(i+1)$-times continuously differentiable on $(x_0,\infty)$ for some $x_0\geq 0,$ is called a \emph{tempered function of degree $i$} (we write $d_t=i$), if the following hold:
\begin{enumerate}
\item[(1)] $t^{(i+1)}(x)$ tends monotonically to $0$ as $x\to\infty;$

\item[(2)]  $\lim_{x\to\infty}x|t^{(i+1)}(x)|=\infty.$
\end{enumerate}
Tempered functions of degree $0$ are called \emph{Fej\'{e}r functions}. 
\end{definition*}

A big difference between Hardy field functions and tempered functions is that in the latter class, limits of ratios may not exist. In order to avoid various problematic cases, we will restrict our study to the following subclass of tempered functions (see \cite{bk}, \cite{koutsogiannis}):

Let $\mathcal{R}:=\Big\{g\in C^\infty(\mathbb{R}^+):\;\lim_{x\to\infty}\frac{xg^{(i+1)}(x)}{g^{(i)}(x)}\in \mathbb{R}\;\;\text{for all}\;\;i\in\mathbb{N}_{0}\Big\};$

$\mathcal{T}_i:=\Big\{g\in\mathcal{R}:\;\exists\;i<\alpha\leq i+1,\;\lim_{x\to\infty}\frac{xg'(x)}{g(x)}=\alpha,\;\lim_{x\to\infty}g^{(i+1)}(x)=0\Big\};$ 

and $\mathcal{T}:=\bigcup_{i=0}^\infty \mathcal{T}_i.$

It is known that every function $t\in \mathcal{T}_i$ is a tempered function of degree $d_t=i$ and satisfies the growth condition: $x^i\log x\prec t(x)\prec x^{i+1}$ (see \cite{bk}).

\

We will show that our method applies to more general iterates. In particular, we will deal with functions of the form $a=c_1h+c_2t,$ where $(c_1,c_2)\in \mathbb{R}^2\setminus\{(0,0)\},$ $h=s_h+p_h+e_h\in \mathcal{H},$ a Hardy field function of polynomial growth, and $t\in \mathcal{T},$ a tempered function with $\max\{\log, c_2t\}\prec s_h$ or $c_1s_h\prec t$ (notice that the latter also covers the case when $h$ is a polynomial function).

\begin{theorem}\label{t1_3}
	Let $(X,\mathcal{B},\mu,T_{1},\dots,T_{d})$ be a system with commuting and invertible transformations, and $a$ be a function of the form $a=c_1h+c_2t,$ where $(c_1,c_2)\in \mathbb{R}^2\setminus\{(0,0)\},$ $h\in \mathcal{H},$ a Hardy field function of polynomial growth, and $t\in \mathcal{T},$ a tempered function, that satisfy:
 \begin{itemize}
     \item[{\bf{Case 1.}}]  $c_1c_2\neq 0,$ $t\prec s_h,$ $\lim_{x\to\infty}\frac{xs_h^{(d_{s_h}+1)}(x)}{s_h^{(d_{s_h})}(x)}\neq 0$ and 
$\lim_{x\to\infty}\frac{xs_h^{(d_{s_h}+2)}(x)}{s_h^{(d_{s_h}+1)}(x)}\neq 0$;\footnote{ These cannot be simultaneously $0$ by L'H\^opital's rule, but if $s_{h_1}(x):=\log^2 x\prec s_{h_2}(x):=x/\log x,$ then we have $d_{s_{h_1}}=d_{s_{h_2}}=0$ and  $\lim_{x\to\infty}\frac{xs_{h_1}'(x)}{s_{h_1}(x)}=\lim_{x\to\infty}\frac{xs_{h_2}''(x)}{s_{h_2}'(x)}=0.$} or

\item[{\bf{Case 2.}}] $c_2\neq 0,$ $s_h\prec t,$ and  $\lim_{x\to\infty}\frac{xt^{(d_{t}+2)}(x)}{t^{(d_{t}+1)}(x)}\neq 0$;\footnote{ Notice that in this case $h$ can be any polynomial function.}  or








\item[{\bf{Case 3.}}] $c_2=0,$ and $\log \prec s_h.$
 \end{itemize}
  Then  $(T_{1}^{[a(n)]})_n,\dots,(T_{d}^{[a(n)]})_{n}$ are jointly ergodic for $\mu$ if, and only if, both of the following conditions are satisfied:
	\begin{itemize}
		\item[(i)] $((T_{i}T^{-1}_{j})^{[a(n)]})_{n}$ is ergodic for $\mu$ for all $1\leq i,j\leq d,$ $i\neq j$; and
		\item[(ii)] $((T_{1}\times \dots\times T_{d})^{[a(n)]})_{n}$ is ergodic for $\mu^{\otimes d}$.
	\end{itemize}
\end{theorem}

\begin{remark}
Theorem~\ref{t1_2} corresponds to Case 3, so it remains to show the first two.

The function 
\[ a_1(x)=x^{\pi}/ \log x+x^{1/2}(2+\cos\sqrt{\log x})\] is covered by Theorem~\ref{t1_3}, Case 1, but it is not covered by Theorem~\ref{t1_2}.\footnote{ $x^{1/2}(2+\cos\sqrt{\log x})$ is not a Hardy field function (see \cite{bk}), so $a$ is not Hardy as well.} 

The function
\[ a_2(x)=x^{17}+x^{1/2}(2+\cos\sqrt{\log x})\] is covered by Theorem~\ref{t1_3}, Case 2, but it is not covered by Theorem~\ref{t1_2}.
\end{remark}

As with Conjecture~\ref{conj}, we actually expect a more general result to hold, for multiple functions.

\begin{conjecture}\label{conj1}
	Let $(X,\mathcal{B},\mu,T_{1},\dots,T_{d})$ be a system and $\log \prec a_i=c^1_ih_i+c_i^2t_i,$ $1\leq i\leq d,$ where, for each $1\leq i\leq d,$ we have $(c_i^1,c_i^2)\in\mathbb{R}^2\backslash\{(0,0)\}$, $h_i\in \mathcal{H}$ and  $t_i\in\mathcal{T}$ such that $h_{i}$ and $t_{i}$ have different polynomial growth rates, 
 satisfying some ``standard assumptions''.\footnote{ What we mean here are the usual additional assumptions we have to postulate on the growth rates of the functions (e.g., as the ones in Theorem~\ref{t1_3}) in order to avoid local obstructions.} Then $(T_{1}^{[a_1(n)]})_n,\dots,(T_{d}^{[a_d(n)]})_{n}$ are jointly ergodic for $\mu$ if, and only if, both of the following conditions are satisfied:
	\begin{itemize}
		\item[(i)] $\big(T_{i}^{[a_i(n)]}T_{j}^{-[a_j(n)]}\big)_{n}$ is ergodic for $\mu$ for all $1\leq i,j\leq d,$ $i\neq j$; and
		\item[(ii)] $\big(T_{1}^{[a_1(n)]}\times \dots\times T_{d}^{[a_d(n)]}\big)_{n}$ is ergodic for $\mu^{\otimes d}$.
	\end{itemize}
\end{conjecture}


To show Theorem~\ref{t1_3}, following the proof of Theorem~\ref{t1_2}, we have to show that the function $a$ satisfies the conclusion of Theorem~\ref{T:upper bound} for some appropriate function $L$. This last part, i.e., to find such a function $L,$ is trickier in Case 2 than the other two cases, as we have to adapt the proof of \cite[Lemma~3.3]{Ts} to work for positive Fej\'er functions. We also show that the geometric mean of two special Fej\'er functions, is a Fej\'er function as well (see Lemma~\ref{L:gmf} below).

\subsection{Case 1 of Theorem~\ref{t1_3}} Dropping the $e_h(x)$ term from the function $a(x)=s_h(x)+p_h(x)+e_h(x)+t(x),$ setting $b(x):=s_h(x)+t(x),$  it suffices to deal with functions of the form $a(x)=b(x)+p_h(x).$ 

\

\noindent Setting $K:=\max\{d_{s_h},d_{p_h}\}+1,$ we will show 
\[1\prec\vert b^{(K)}(x)\vert^{-\frac{1}{K}}\prec \vert b^{(K+1)}(x)\vert^{-\frac{1}{K+1}}\prec x,\]
and that $b^{(K+1)}$ is (eventually) monotone.

Using the additional assumptions about the function $s_h,$ for every $k\in \mathbb{N},$ we have 
\[\frac{b^{(k)}(x)}{s_h^{(k)}(x)}=1+\frac{t^{(k)}(x)}{s_h^{(k)}(x)}=1+\left(\prod_{i=1}^{k}\frac{xt^{(i)}(x)}{t^{(i-1)}(x)}\cdot \frac{s_h^{(i-1)}(x)}{xs_h^{(i)}(x)}\right)\cdot \frac{t(x)}{s_h(x)}\to 1,\]
hence, it suffices to show that
\[1\prec\vert s_h^{(K)}(x)\vert^{-\frac{1}{K}}\prec \vert s_h^{(K+1)}(x)\vert^{-\frac{1}{K+1}}\prec x,\]
and that $s_h^{(K+1)}$ is (eventually) monotone. 

As in the proof of Theorem~\ref{t1_2}, since $K\geq d_{s_h}+1,$ the required relation follows by \cite[Proposition~A.2]{Ts}. Monotonicity is immediate as $s_h^{(K+1)}$ is a Hardy field function.

\begin{proof}[Proof of Theorem~\ref{t1_3}, Case 1]
For $K=\max\{d_{s_h},d_{p_h}\}+1,$ it suffices to show that $a$ satisfies the conclusion of Theorem~\ref{T:upper bound}. Using Remark~\ref{R:5.1}, it suffices to show that there exists a positive function $L\in \mathcal{H}$ such that 
\[1\prec\vert b^{(K)}(x)\vert^{-\frac{1}{K}}\prec L(x)\prec \min\big\{\vert b^{(K+1)}(x)\vert^{-\frac{1}{K+1}},\vert b^{(K)}(x)\vert^{-\frac{K+1}{K^{2}}}\big\}.\]
By the previous discussion, this is possible as we can pick a positive $L\in \mathcal{H}$ so that 
\[1\prec\vert s_h^{(K)}(x)\vert^{-\frac{1}{K}}\prec L(x)\prec \min\big\{\vert s_h^{(K+1)}(x)\vert^{-\frac{1}{K+1}},\vert s_h^{(K)}(x)\vert^{-\frac{K+1}{K^{2}}}\big\}.\]
The proof is now complete.
\end{proof}

\subsection{Case 2 of Theorem~\ref{t1_3}} Analogously to Case 1, dropping the $e_h(x)$ term from the function $a(x)=s_h(x)+p_h(x)+e_h(x)+t(x),$ setting $b(x):=s_h(x)+t(x),$  it suffices to deal with functions of the form $a(x)=b(x)+p_h(x).$ The case $s_h=0$ (and actually $h=0$) follows from the more general $s_h\neq 0,$ hence we assume the latter.

\

\noindent Setting $K:=\max\{d_{p_h}, d_{t}\}+1,$ we will show 
\[1\prec\vert b^{(K)}(x)\vert^{-\frac{1}{K}}\prec \vert b^{(K+1)}(x)\vert^{-\frac{1}{K+1}}\prec x,\]
and that $b^{(K+1)}$ is (eventually) monotone.

Using the additional assumptions about the function $s_h,$ for every $k\in \mathbb{N},$ we have 
\[\frac{b^{(k)}(x)}{t^{(k)}(x)}=\frac{s_h^{(k)}(x)}{t^{(k)}(x)}+1=\left(\prod_{i=1}^{k}\frac{xs_h^{(i)}(x)}{s_h^{(i-1)}(x)}\cdot \frac{t^{(i-1)}(x)}{xt^{(i)}(x)}\right)\cdot \frac{s_h(x)}{t(x)}+1\to 1,\]
hence, it suffices to show that
\[1\prec\vert t^{(K)}(x)\vert^{-\frac{1}{K}}\prec \vert t^{(K+1)}(x)\vert^{-\frac{1}{K+1}}\prec x,\]
and that $t^{(K+1)}$ is (eventually) monotone. 

\medskip

\noindent$\bullet$ $t^{(K)}\prec 1.$

This is true since $t'(x)\ll \frac{t(x)}{x}$ by the definition of the set $\mathcal{R},$ so, iterating this, we get $t^{(k)}(x)\ll \frac{t(x)}{x^k}$ for all $k\in \mathbb{N}$. In particular, $t^{(K)}(x)\to 0$ as $K\geq d_t+1.$
			
\medskip

\noindent$\bullet$ $1/x^K\prec t^{(K)}(x).$ 

Let $K=d_t+1+\rho,$ where $\rho\geq 0.$ Using the additional assumption on $t,$ we have 
\[\frac{|t^{(K)}(x)|}{\frac{1}{x^K}}=x^{d_t}\cdot x|t^{(d_t+1)}(x)|\cdot \prod_{i=d_t+1}^{d_t+\rho}\frac{x|t^{(i+1)}(x)|}{|t^{(i)}(x)|} \to \infty.\]
$\bullet$  $|t^{(K)}(x)|^{-\frac{1}{K}}\prec |t^{(K+1)}(x)|^{-\frac{1}{K+1}}$. 

By the definition of $\mathcal{R},$ we have that
			\[\lim_{x\to\infty}\frac{xt^{(K+1)}(x)}{t^{(K)}(x)} \in \mathbb{R},\] 
			so, using the fact that $1/x^K\prec t^{(K)}(x),$ we get
			\[\left(t^{(K+1)}(x)\right)^K\ll \frac{\left(t^{(K)}(x)\right)^K}{x^K}\prec \left(t^{(K)}(x)\right)^{K+1}.\]
$\bullet$ Monotonicity of $t^{(K+1)}.$ 

By the definition of $\mathcal{R},$ and the additional assumption on $t,$ for all $k\geq d_t+1,$ we have
\[\lim_{x\to\infty}\frac{xt^{(k+1)}(x)}{t^{(k)}(x)}<0.\] 

The following lemma will allow us to use a ``change of variables'' argument when we sum along a positive Fej\'er function (analogous to \cite[Lemma~3.3]{Ts} that we used in the proof of Proposition~\ref{P:HostSemi}). As it is a modification of \cite[Lemma~3.3]{Ts}, we only present a sketch of its proof, following the one of the latter. 

\begin{lemma}\label{L:cov2}
Let $\kappa\in\mathbb{N}$ and consider a two-parameter sequence $(A_{R,n})_{R,n\in\mathbb{N}}$ in a normed space such that $\|A_{R,n}\| \leq 1$ for all $R, n\in\mathbb{N}.$ If for a positive Fej\'er function $L$ we have
\[\limsup_{R\to\infty}\mathbb{E}_{1\leq N\leq R}\norm{ \mathbb{E}_{N\leq n\leq N+L(N)}A_{R,n}}^{\kappa}=0,\]
then
\[\limsup_{R\to\infty}\norm{\mathbb{E}_{1\leq n\leq R}A_{R,n}}=0.\]
\end{lemma}

\begin{proof}[Sketch of the proof]
Since
\[\mathbb{E}_{1\leq N\leq R}\norm{ \mathbb{E}_{N\leq n\leq N+L(N)}A_{R,n}}^{\kappa}\geq \norm{\mathbb{E}_{1\leq N\leq R}\big(\mathbb{E}_{N\leq n\leq N+L(N)}A_{R,n}\big)}^{\kappa},\] the result follows if we show
\[\norm{\mathbb{E}_{1\leq N\leq R}\big(\mathbb{E}_{N\leq n\leq N+L(N)}A_{R,n}\big)-\mathbb{E}_{1\leq n\leq R}A_{R,n}}\to 0,\;\;\text{as}\;R\to\infty.\] 
Let $u$ be the inverse of $L(x)+x.$  We have that $u$ is strictly increasing (since $u(L(x)+x)=x,$ thus $u'(L(x)+x)L'(x)=1$) and that $\lim_{x\to\infty} u(x)=\infty.$ Using the fact that $L(x)\prec x,$ we get
\[\frac{u(x)}{x}=\frac{1}{\frac{L(u(x))}{u(x)}+1}\to 1.\]
We have 
\begin{equation}\label{E:fl}\mathbb{E}_{1\leq N\leq R}\big(\mathbb{E}_{N\leq n\leq N+L(N)}A_{R,n}\big)=\frac{1}{R}\left(\sum_{n=1}^R p(n)A_{R,n}+\sum_{n=R+1}^{R+L(R)} p(n)A_{R,n}\right),\end{equation} where, for large enough $n,$
\[p(n)=\frac{1}{L([u(n)])+1}+\ldots+\frac{1}{L(n)+1}+o_n(1).\footnote{ $o_n(1)$ is a quantity that goes to $0$ as $n\to\infty.$}
\]
As $L$ is strictly increasing, $u^{-1}$ is onto in a half-line of $\mathbb{R},$ and $L'\prec 1,$ following the proof of \cite[Lemma~3.3]{Ts}, we have that $\lim_{n\to\infty}p(n)=1.$ This implies that $(p(n))_n$ is bounded, so, from the fact that $L(x)\prec x,$ using also that $(A_{R,n})_{R,n}$ is bounded in norm, we get
\[\norm{\sum_{n=R+1}^{R+L(R)} p(n)A_{R,n}}=o_R(1),\] and (also because $\lim_{n\to\infty}p(n)=1$)
\[\norm{\frac{1}{R}\sum_{n=1}^R p(n)A_{R,n}-\frac{1}{R}\sum_{n=1}^R A_{R,n}}\leq\frac{1}{R}\sum_{n=1}^R|p(n)-1|=o_R(1),\] hence, the result follows from \eqref{E:fl}.
\end{proof}

\begin{lemma}[Section 3, I (iii), \cite{koutsogiannis}]\label{L:v}
Let $t\in \mathcal{T}_{d_t},$ with $\lim_{x\to\infty}\frac{xt'(x)}{t(x)}=\alpha\in (d_t,d_t+1).$ Then, for every $\beta<\alpha<\gamma,$ we have that
\[x^{\beta}\prec t(x)\prec x^{\gamma}.\]
\end{lemma}

\begin{proof}
For every $0<\varepsilon< \min\{\alpha-d_t,d_t+1-\alpha\}$ there exists $M>0$ such that $\alpha-\varepsilon<\frac{xt'(x)}{t(x)}<\alpha+\varepsilon$ for all $x>M,$ hence \[\log\left(\frac{x}{M}\right)^{\alpha-\varepsilon}=\int_M^x \frac{\alpha-\varepsilon}{x}\;dx\leq \log\frac{|t(x)|}{|t(M)|}=\int_{M}^x \frac{t'(x)}{t(x)}\;dx\leq \int_M^x \frac{\alpha+\varepsilon}{x}\;dx=\log\left(\frac{x}{M}\right)^{\alpha+\varepsilon},\] 
from where we get
\[x^{\alpha-\varepsilon}\ll t(x)\ll x^{\alpha+\varepsilon}.\] As $\varepsilon>0$ can be taken arbitrarily small, we have the conclusion.
\end{proof}

\begin{remark}
Notice that, for a function $t\in \mathcal{T}_{d_t},$  $\lim_{x\to\infty}\frac{xt'(x)}{t(x)}=\alpha\in (d_t,d_t+1)$ is equivalent to $\lim_{x\to\infty}\frac{xt^{(d_t+2)}(x)}{t^{(d_t+1)}(x)}\neq 0.$ 
\end{remark}

We also need the following lemma which implies that the geometric mean of two specific Fej\'er functions is a Fej\'er function as well.

\begin{lemma}\label{L:gmf}
Let $t\in \mathcal{T}_{d_t}$ with $\lim_{x\to\infty}\frac{xt'(x)}{t(x)}=\alpha\in (d_t,d_t+1).$ For any $K\geq d_t+1$ with $\frac{K^2+K}{2K+1}\neq \alpha,$  we can find a positive Fej\'er function $L$ such that
\begin{equation}\label{E:sss}
\vert t^{(K)}(x)\vert^{-\frac{1}{K}}\prec L(x)\prec \min\big\{\vert t^{(K+1)}(x)\vert^{-\frac{1}{K+1}},\vert t^{(K)}(x)\vert^{-\frac{K+1}{K^{2}}}\big\}.
\end{equation}
\end{lemma}

\begin{proof}
First step is to show that (eventually) one of the two functions on the right hand side of the relation of the statement is the minimum one. To this end, notice that 
\begin{equation}\label{E:cr}\frac{\vert t^{(K+1)}(x)\vert^{-\frac{1}{K+1}}}{\vert t^{(K)}(x)\vert^{-\frac{K+1}{K^{2}}}}=\left(\frac{x|t^{(K+1)}(x)|}{|t^{(K)}(x)|}\right)^{-\frac{1}{K+1}}\cdot \left(x\vert t^{(K)}(x)\vert^{\frac{2K+1}{K^2}}\right)^{\frac{1}{K+1}}.\end{equation}

\noindent {\bf{Case I.}} $\frac{K^2+K}{2K+1}<\alpha.$

Using Lemma~\ref{L:v}, we have $x^{\frac{K^2+K}{2K+1}}\prec t(x),$ so $x^{-\frac{K^2}{2K+1}}\prec t^{(K)}(x).$ The latter follows by the relation \[\frac{x^{-\frac{K^2}{2K+1}}}{t^{(K)}(x)}=\frac{x^{\frac{K^2+K}{2K+1}}}{t(x)}\cdot\prod_{i=1}^K \frac{t^{(i-1)}(x)}{xt^{(i)}(x)}.\] \eqref{E:cr} implies
\[\min\big\{\vert t^{(K+1)}(x)\vert^{-\frac{1}{K+1}},\vert t^{(K)}(x)\vert^{-\frac{K+1}{K^{2}}}\big\}=\vert t^{(K)}(x)\vert^{-\frac{K+1}{K^{2}}}.\]

\
Assuming that $t^{(K)}>0$ (the other case is similar), we define \[L(x):=t^{(K)}(x)^{-\frac{2K+1}{2K^2}}\] (i.e., we took the geometric mean of the two functions appearing in \eqref{E:sss}). 

Notice that for every $k\geq K$ we have 
\begin{equation*}\label{E:dnz}
\lim_{x\to\infty}\frac{xt^{(k+1)}(x)}{t^{(k)}(x)}< 0,
\end{equation*}
i.e., $t^{(k+1)}$ has (eventually) the opposite sign of $t^{(k)}.$

For the function $L$ we have:
\[\lim_{x\to\infty}\frac{xL'(x)}{L(x)}=-\frac{2K+1}{2K^2}\cdot\lim_{x\to\infty}\frac{xt^{(K+1)}(x)}{t^{(K)}(x)}>0,\] so, $xL'(x)\to \infty,$ and \[\lim_{x\to\infty}L'(x)=\lim_{x\to\infty}\frac{xL'(x)}{L(x)}\cdot\lim_{x\to\infty}\frac{L(x)}{x}\to 0.\]
We also have to show the monotonicity of $L'.$ 
We have 
\[\frac{xL''(x)}{L'(x)}=-\frac{2K^2+2K+1}{2K^2}\cdot \frac{xt^{(K+1)}(x)}{t^{(K)}(x)}+\frac{xt^{(K+2)}(x)}{t^{(K+1)}(x)}.\]
Using the assumption on $\alpha$ and $K,$ we get
\[\lim_{x\to\infty}\frac{xL''(x)}{L'(x)}=-\frac{2K^2+2K+1}{2K^2}\cdot(\alpha-K)+\alpha-(K+1)<0.\]
To sum up, we have that $L$ is a (positive) Fej\'er function as was to be shown.

\

\noindent {\bf{Case II.}} $\alpha<\frac{K^2+K}{2K+1}.$

Using Lemma~\ref{L:v}, we have $t(x) \prec x^{\frac{K^2+K}{2K+1}},$ so, as in Case I, $t^{(K)}(x)\prec x^{-\frac{K^2}{2K+1}}.$ \eqref{E:cr} implies
\[\min\big\{\vert t^{(K+1)}(x)\vert^{-\frac{1}{K+1}},\vert t^{(K)}(x)\vert^{-\frac{K+1}{K^{2}}}\big\}=\vert t^{(K+1)}(x)\vert^{-\frac{1}{K+1}}.\] 

\
Assuming that $t^{(K)}>0$ (the other case is similar), we define \[L(x):=t^{(K)}(x)^{-\frac{1}{2K}}\cdot \big(-t^{(K+1)}(x)\big)^{-\frac{1}{2(K+1)}}\] (i.e., we are taking yet again the geometric mean of the two functions in \eqref{E:sss}). 


Since
\[\frac{xL'(x)}{L(x)}=-\frac{1}{2K}\cdot\frac{xt^{(K+1)}(x)}{t^{(K)}(x)}-\frac{1}{2(K+1)}\cdot\frac{xt^{(K+2)}(x)}{t^{(K+1)}(x)},\] using the assumption on $\alpha$ and $K,$ we have 
\[\lim_{x\to\infty}\frac{xL'(x)}{L(x)}  = -\frac{1}{2K}(\alpha-K)-\frac{1}{2(K+1)}(\alpha-K-1)>0,
\]
hence, as in the previous case, $xL'(x)\to\infty$ and $L'(x)\to 0.$

To finish the proof of the statement, it suffices to show that $L'$ is (eventually) monotone. To this end, we compute 
\[\frac{L''(x)}{L'(x)}=\frac{L'(x)}{L(x)}-\frac{\sum_{i\in\{0,1\}}\frac{1}{2(K+i)}\cdot \frac{t^{(K+2+i)}(x)\cdot t^{(K+i)}(x)-\big(t^{(K+1+i)}(x)\big)^2}{\big(t^{(K+i)}(x)\big)^2}}{\frac{L'(x)}{L(x)}},\]
hence, setting $\beta:=\alpha-K$ (notice that, since $\alpha>0,$ we have  $\beta>-K$), we have
\begin{eqnarray*}\lim_{x\to\infty}\frac{xL''(x)}{L'(x)} & = &-\frac{1}{2K}\beta-\frac{1}{2(K+1)}(\beta-1)+\frac{\sum_{i\in \{0,1\}}\frac{1}{2(K+i)}\big((\beta-i-1)(\beta-i)-(\beta-i)^2\big)}{\frac{1}{2K}\beta+\frac{1}{2(K+1)}(\beta-1)}\\
& = & -\frac{1}{2K}\beta-\frac{1}{2(K+1)}(\beta-1)-1<0,
\end{eqnarray*}
from where the conclusion follows.
\end{proof}

We are now ready to complete the proof of Case 2 of Theorem~\ref{t1_3}.

\begin{proof}[Proof of Theorem~\ref{t1_3}, Case 2]
Assuming that $\lim_{x\to\infty}\frac{xt'(x)}{t(x)}=\alpha\in (d_t,d_t+1),$ let $K\geq \max\{d_{p_h},$ $ d_t\}+1$ with $\frac{K^2+K}{2K+1}\neq \alpha.$ Using Lemma~\ref{L:cov2} in place of \cite[Lemma~3.3]{Ts}, we have that the corresponding statement to Proposition~\ref{P:HostSemi} holds if $L$ is a positive Fej\'er function. Thus, to show the statement, it suffices to show that $a$ satisfies the conclusion of the corresponding to Theorem~\ref{T:upper bound} result for such a function $L.$

Analogously to Case 1 (see Remark~\ref{R:5.1}), it suffices to show that we can find a positive Fej\'er function $L$ such that 
\[1\prec\vert b^{(K)}(x)\vert^{-\frac{1}{K}}\prec L(x)\prec \min\big\{\vert b^{(K+1)}(x)\vert^{-\frac{1}{K+1}},\vert b^{(K)}(x)\vert^{-\frac{K+1}{K^{2}}}\big\}.\]
This is possible, as Lemma~\ref{L:gmf} implies that  
\[1\prec\vert t^{(K)}(x)\vert^{-\frac{1}{K}}\prec L(x)\prec \min\big\{\vert t^{(K+1)}(x)\vert^{-\frac{1}{K+1}},\vert t^{(K)}(x)\vert^{-\frac{K+1}{K^{2}}}\big\},\]
for some positive Fej\'er function $L.$
\end{proof}

\noindent {\bf{Acknowledgements.}} We mostly thank Andreu Ferr\'e Moragues, with whom we started discussing this project. The second author wants to express his gratitude towards the Center for Mathematical Modeling of the University of Chile and Virginia Tech University for their hospitality. The third author also thanks the Center for Mathematical Modeling of the University of Chile for its hospitality.

\end{document}